\documentclass[11pt]{article}
\usepackage{xcolor}
\usepackage{geometry}
\usepackage[full]{textcomp}
\usepackage{amssymb,amsmath,amsfonts,latexsym}
\usepackage{amsmath,graphicx,bm,xcolor,url}
\usepackage[caption=false]{subfig}
\usepackage{verbatim}
\usepackage{textcomp}
\usepackage{mathrsfs}
\usepackage[colorlinks=true,
            linkcolor=blue,
            citecolor=blue,
            urlcolor=blue]{hyperref}
\usepackage{relsize}
\usepackage{amsthm}
\usepackage[nameinlink,capitalize]{cleveref}
\usepackage{algorithm}
\usepackage{algorithmic}
\usepackage{float}
\usepackage{bbm}

\usepackage{tabularx}
\newcolumntype{L}{>{\raggedright\arraybackslash}X}%
\newcolumntype{R}{>{\raggedleft\arraybackslash}X}%
\newcolumntype{C}{>{\centering\arraybackslash}X}%
\catcode`~=11 \def\UrlSpecials{\do\~{\kern -.15em\lower .7ex\hbox{~}\kern .04em}} \catcode`~=13 

\allowdisplaybreaks[3]

\newcommand{\nn}{\nonumber}

\newcommand{\ba}{\mathbf{a}}
\newcommand{\bA}{\mathbf{A}}
\newcommand{\bb}{\mathbf{b}}

\newcommand{\be}{\mathbf{e}}

\newcommand{\bx}{\mathbf{x}}
\newcommand{\bX}{\mathbf{X}}

\newcommand{\bepsilon}{\bm{\epsilon}}

\newcommand{\truncZsq}[1][\alpha_0]{\mathbb{E}\left[Z^2 \mathbf{1}_{|Z| \leq \Phi_{#1}} \right]}

\newtheorem{assumption}{Assumption}
\newtheorem{corollary}{Corollary}
\newtheorem{definition}{Definition} 
\newtheorem{remark}{Remark}

\newcommand{\qednew}{\nobreak \ifvmode \relax \else
      \ifdim\lastskip<1.5em \hskip-\lastskip
      \hskip1.5em plus0em minus0.5em \fi \nobreak
      \vrule height0.75em width0.5em depth0.25em\fi}

\newcommand{\D}{\operatorname{D}_{KL}}

\renewcommand{\bm}{\mathbf}

\renewcommand{\boldsymbol}{\mathbf}
\newtheorem{thm}{Theorem}

\newtheorem{lem}{Lemma}
\newtheorem{prop}{Proposition}

\crefname{thm}{Theorem}{Theorems}
\crefname{defn}{Definition}{Definitions}
\crefname{lem}{Lemma}{Lemmas}
\crefname{prop}{Proposition}{Propositions}
\crefname{Cor}{Corollary}{Corollaries}
\crefname{ass}{Assumption}{Assumptions}
\crefname{re}{Remark}{Remarks}

\newcommand{\target}{\bm{x}^*}
\newcommand{\init}{\bm{x}_{0}}

\newcommand{\epsl}{\epsilon_l}
\newcommand{\epsu}{\epsilon_u}

\crefname{equation}{}{}

\newcommand{\updateid}[2][0]{#2,#1}

\crefname{algocf}{Algorithm}{Algorithms}

\renewcommand{\ln}{\log}
\newcommand{\subsample}[2]{\{1, \dots, D\}}

\newcommand{\Bin}[2]{{\rm Bin}(#1,#2)}

\newcommand{\ur}{\frac{q + \epsu}{1 - \beta}}
\newcommand{\lr}{\frac{q-\beta-\epsl}{1 - \beta}}
\newcommand{\lrr}{\frac{\alpha}{1-\beta}}
\newcommand{\urr}{\frac{\alpha'}{1-\beta}}

\newcommand{\cv}{\xi}  %
\newcommand{\fobv}{f_{\operatorname{obl}}}

\newcommand{\cobl}{c_{\operatorname{obl}}}
\newcommand{\cadv}{c_{\operatorname{adv}}}

\newif\ifshow %
\newif\ifshowwt
\showwttrue

\newcommand{\wt}[2][f]{%
  \ifshowwt
    \ifnum\pdfstrcmp{#1}{i}=0%
      \textcolor{purple}{wt: #2}%
    \else%
      \footnote{ \textcolor{purple}{wt:#2}
      }%
    \fi%
  \else%
  \fi%
}

\showtrue  
\newcommand{\jr}[2][f]{%
  \ifshow
    \ifnum\pdfstrcmp{#1}{i}=0%
      \textcolor{blue}{#2}%
    \else%
      \footnote{
      \textcolor{blue}{#2}
      }%
    \fi%
  \else%
  \fi%
}
\newcommand{\KL}[2]{D_{\mathrm{KL}}\left(#1\,\|\,#2\right)}

\providecommand{\nn}{\nonumber}

\providecommand{\ba}{\mathbf{a}}\providecommand{\bA}{\mathbf{A}}\providecommand{\bb}{\mathbf{b}}\providecommand{\be}{\mathbf{e}}\providecommand{\bx}{\mathbf{x}}\providecommand{\bX}{\mathbf{X}}

\providecommand{\bepsilon}{\bm{\epsilon}}

\makeatletter
\@ifundefined{independenT}{}{}
\@ifundefined{indep}{}{}
\makeatother

\providecommand{\qednew}{\nobreak \ifvmode \relax \else
      \ifdim\lastskip<1.5em \hskip-\lastskip
      \hskip1.5em plus0em minus0.5em \fi \nobreak
      \vrule height0.75em width0.5em depth0.25em\fi}

\graphicspath{{../}}

\usepackage[nomath]{stix}
\usepackage{epstopdf}
\newtheorem*{question*}{Question}
\newtheorem*{theorem*}{Theorem}
\title{Quantile Randomized Kaczmarz for Streaming  Linear \\ Systems with Massart Noise%
} 
\author{
Emeric Battaglia\thanks{University of California, Irvine.}\hspace{3em}
Jian-Feng Cai\thanks{Hong Kong University of Science and Technology.}\hspace{3em}
Junren Chen\thanks{Columbia University.}\\[8pt]
Anna Ma\footnotemark[1]\hspace{3em}
Deanna Needell\thanks{University of California, Los Angeles.}\hspace{3em}
Tong Wu\footnotemark[2]
}

\date{\today}

\begin{document}

\maketitle

\begin{abstract}
   Quantile randomized Kaczmarz (QRK) has proven to be an efficient solver for corrupted linear systems and has received much attention. 
   It was recently shown by Cai et al. (SIAM J. Matrix Anal. Appl. 47(2):802--823, 2026) that using $O(\frac{\log T}{\log(1/\beta)})$ samples for computing the quantile is necessary and sufficient for QRK to converge linearly over $T$ iterations when solving linear systems with a $\beta$-fraction of arbitrary corruptions, as long as $\beta$ is small enough. 
   However, it remains unclear how large the corruption level $\beta$ can be, and how to compute the required subsample size $D$ explicitly, without hidden constants.
  This paper studies streaming linear systems with Massart noise via QRK using an order-optimal batch size $D=O(\log T)$ in each update. 
  The independence of samples from previous iterations in the streaming setting enables a sharper analysis, yielding explicit, computable bounds on both the tolerable corruption level and the required subsample size. 
  In particular, we establish linear convergence for corruption levels of up to approximately 7\%.
  We also discuss how the constants improve under oblivious noise. %
\end{abstract}

\section{Introduction}
\label{sec11}

Solving large-scale corrupted linear systems arises in scientific computing, data science, and signal processing \cite{boyd2004convex,medical_imaging_2,sensornet}. The goal is to recover the true solution $\target\in\mathbb{R}^n$ from the measurement matrix $\bA = [\ba_1, \ldots, \ba_m]^\top \in \mathbb{R}^{m \times n}$ and the observations
\begin{equation}\label{eq: problem}
    \bb=\bA\target + \bepsilon,
\end{equation}
where the corruption $\bepsilon$ is an arbitrary $(\beta m)$-sparse vector that may be adversarially generated. Traditionally, one seeks the least squares solution from noisy observations. However, since the nonzero values of $\bepsilon$ can be arbitrarily large, the least squares solution may not approximate $\target$ well. Instead of seeking the least squares solution, we assume that $\beta \in (0,1)$ is small enough so that perfect recovery of $\target$ remains feasible.

In an overdetermined setting where the number of measurements $m$ is much larger than the dimension $n$, Kaczmarz methods \cite{karczmarz1937angenaherte} are attractive because they require access to only one row per iteration. %
Specifically, one chooses a row $j$ and updates the current iterate, say $\bx_k$, by projecting it onto the solution hyperplane associated with the $j$-th row: 
\begin{align}
    \bx_{k+1}
    =
    \bx_k
    -
    \frac{\ba_j^\top \bx_k-b_j}{\|\ba_j\|^2}\ba_j.\label{rkupdate}
\end{align}
It was shown \cite{karczmarz1937angenaherte,herman1993algebraic} that the convergence rate may depend on the order of row selection. 
By sampling rows with probabilities proportional to their squared norms, randomized Kaczmarz (RK), introduced by Strohmer and Vershynin \cite{strohmer2009randomized}, enjoys strong theoretical guarantees for consistent linear systems and converges linearly to the true solution. See also \cite{needell2010randomized} for stability under bounded noise. Nonetheless, RK is not automatically robust to sparse adversarial corruption and in general does not solve \cref{eq: problem}.

Haddock, Needell, Rebrova, and Swartworth \cite{haddock_quantilebased_2022} introduced quantile randomized Kaczmarz (QRK), which computes the quantile of the residuals of a size-$D$ subsample\footnote{Throughout the paper, we use $D$ to denote the quantile subsample size, namely the number of rows required in one iteration of QRK to compute the quantile.} at each RK iteration and then performs \cref{rkupdate} only if the residual of $\bx_k$ at row $j$, i.e.\ $|\ba_j^\top \bx_k - b_j|$, is below the quantile. The authors proved that QRK converges linearly to $\target$, as RK does for consistent linear systems, under a class of random matrices $\bA$ (cf. Assumptions 1--2 therein). Steinerberger \cite{steinerberger2023quantile} adapted the analysis to a deterministic matrix $\bA$ 
and, based on a random matrix heuristic, showed that QRK can handle a corruption level of $\beta\approx 0.005$ under $\ba_i \stackrel{{\rm iid}}{\sim} {\rm Unif}(S^{n-1})$. However, the QRK algorithms analyzed in \cite{haddock_quantilebased_2022,steinerberger2023quantile} rely on quantiles computed from the full sample, requiring access to all $m$ samples at each iteration and thereby eliminating the computational advantage of Kaczmarz methods.

In recent work \cite{cai2025subsamplesizequantilebasedrandomized}, a subgroup of the current authors showed that the QRK method in \cite[Method 2.1]{haddock_quantilebased_2022} with a subsample size $D=O(\frac{\log T}{\log(1/\beta)})$ linearly converges to $\target$ over the first $T$ iterations when applied to corrupted linear systems \cref{eq: problem}. This reduces to $D=O(\log T)$ when $\beta$ is a small universal constant, and further implies that $D=O(\log n)$ suffices if $T=n^{O(1)}$,\footnote{This iteration count is often more than sufficient in practice. For instance, under $\ba_i\sim {\rm Unif}(S^{n-1})$, QRK linearly converges to $\target$ with per-iteration contraction rate $c/n$, and therefore $T=n^2$ yields an in-expectation accuracy of $\exp(-cn)$.} thus achieving a massive reduction from $D=m$ in \cite{haddock_quantilebased_2022,steinerberger2023quantile}. The authors also established a matching lower bound showing that $D=\Omega(\frac{\log T}{\log(1/\beta)})$ is necessary for horizon-$T$ convergence of QRK. We note in passing that an earlier work \cite{haddock_subsampled_2023} toward this goal showed that $D=\alpha m$ (for some $\alpha<1$) is enough to ensure linear convergence of QRK.

As such, \cite{cai2025subsamplesizequantilebasedrandomized} appears to provide the best known theoretical result for finite-horizon QRK. We now briefly recap the main result of \cite{cai2025subsamplesizequantilebasedrandomized}, suppressing some inessential details.
\begin{theorem*}[{\cite[Theorem 1]{cai2025subsamplesizequantilebasedrandomized}}]
Suppose that $\bx_T$ is obtained by the QRK in \cite[Algorithm 1]{cai2025subsamplesizequantilebasedrandomized} with data $(\bA,\bb)$, initialization $\bx_0$, quantile parameter $q\in(0,1)$, quantile subsample size $D$, and iteration number $T$. Under $\ba_i\stackrel{{\rm iid}}{\sim}{\rm Unif}(S^{n-1})$ and $m\gtrsim n$, there exist positive constants $c_1,C_2,c_3$ depending only on $q$ such that, if
\[
\beta<c_1
\qquad\text{and}\qquad
D\ge \frac{C_2\log T}{\log(1/\beta)},
\]
then with high probability,
\[
\|\bx_T-\target\|^2\le \Bigl(1-\frac{c_3}{n}\Bigr)^T\|\bx_0-\target\|^2.
\]
\end{theorem*}

In the theorem, the constant $c_1$ must be sufficiently small. How small should it be? This is an interesting question because it clarifies how much corruption QRK can tolerate. A similar question was asked by Steinerberger \cite[p.~453]{steinerberger2023quantile}; as mentioned above, $c_1=0.005$ was obtained from a random matrix heuristic. However, this appears conservative and relies on full-sample quantiles (i.e., $D=m$). To the best of our knowledge, no explicit value of $c_1$ is known for QRK with $D\ll m$.
Similarly, $C_2$ in the above theorem must be sufficiently large, but it remains unclear how to compute an explicit $C_2$---or, equivalently, an explicit subsample size $D$---that guarantees convergence of QRK.
Determining the tolerable corruption level and an explicit order-minimal subsample size for QRK are among the questions we address in this paper. 

To this end, we propose and analyze QRK for streaming corrupted linear systems,
which serves as a useful surrogate for a highly overdetermined ensemble setting.
In this setting, fresh samples are drawn at each iteration subject to potential corruption---here modeled under the widely studied Massart and oblivious noise models (see \cref{sec:corruption models} and, e.g., \cite{diakonikolas2021relu,diakonikolas2024online,diakonikolas2025online,jeong2025stochastic}).
Compared with the ensemble setting, in which measurements are repeatedly sampled from a fixed matrix $\bA$ \cite{haddock_quantilebased_2022,steinerberger2023quantile,cai2025subsamplesizequantilebasedrandomized}, the independence of fresh samples from previous iterations enables a sharper analysis through explicit expectations (see \cref{rem:streaming-sharper}).
This allows us to address both questions above by obtaining explicit, computable bounds on the tolerable corruption level and the required subsample size.
Our main contributions are twofold: 
\begin{enumerate}
\item \textbf{Linear convergence and maximum tolerable corruption level.}
    We prove that streaming QRK converges linearly to the true solution using only $O(\log T)$ fresh samples per iteration; see \cref{thm: convergence in expectation-streaming}. 
    Building on this convergence guarantee, we characterize the maximum tolerable corruption level under Massart noise in \cref{prop: maximum corruption level}, which reaches $\beta^*(0.85)\approx0.069$; see \cref{fig:q_beta_max_massart}. This substantially improves upon the $\beta\approx0.005$ guarantee of \cite{steinerberger2023quantile}.

    \item \textbf{Explicit and computable subsample size.}
    For corruption levels below the maximum tolerable level, we make the required subsample size $D$ explicit and computable; see \cref{prop: explicit subsample size}. For instance, under a Massart noise rate $\beta = 0.01$, streaming QRK with $q=0.75$ and $D=25$ samples to compute the quantile in each iteration is guaranteed to converge over $T=20000$ iterations; see \cref{fig:D_beta_curve_massart}. While we believe $D=25$ remains quite conservative, to our knowledge a guarantee under such an explicit and reasonably small subsample size was not available in the literature before our work.

\end{enumerate}

Our results can be strengthened further under oblivious noise, in which each corrupted row is perturbed by additive noise independent of everything else; see \cref{sec:corruption models} for details. Under this model, the maximum tolerable corruption level reaches $\beta^*_{\mathrm{obl}}(0.65)\approx0.32$; see \cref{fig:q_beta_max_oblivious}. This substantially improves upon the $\beta\approx0.069$ guarantee under Massart noise, while still requiring only $O(\log T)$ fresh samples per iteration. Moreover, when $q=0.75$ and $\beta=0.01$, choosing $D=13$ guarantees convergence over $T=20000$ iterations; see \cref{fig:D_beta_curve_oblivious}. These stronger guarantees shed light on the question posed in \cite[p.~454]{steinerberger2023quantile}: ``what is the maximum percentage of random corruption that QRK can absorb?''

The remainder of the paper is organized as follows.
\Cref{sec:prelim} collects basic notation and facts, and formalizes streaming QRK together with the Massart and oblivious noise models.
\Cref{sec: streaming setting} presents the main theoretical guarantees under Massart noise and the corresponding results under oblivious noise.
\Cref{sec: numerical results} describes how to compute the explicit thresholds $\beta^*$ and $D^*$ and compares the resulting bounds with numerical simulations.
We conclude in \cref{sec: conclusion}.

\section{Preliminaries}
\label{sec:prelim}

We first collect the basic notation used throughout the paper. We write $[m]=\{1,\dots,m\}$. The inner product and Euclidean norm are $\langle \ba,\bb\rangle=\ba^\top\bb$ and $\|\ba\|=\sqrt{\langle \ba,\ba\rangle}$. We use the natural logarithm $\log(\cdot)$. For Bernoulli parameters $p,q\in[0,1]$, the Kullback-Leibler divergence is
\[
    \KL{p}{q}
    =
    p\log \bigl(\tfrac{p}{q}\bigr)
    +(1-p)\log \bigl(\tfrac{1-p}{1-q}\bigr).
\]
Constants $C,C_i,c,c_i$ may change from line to line. We write $I_1=O(I_2)$ if $I_1\le CI_2$ for an absolute constant $C$, and $I_1=\Omega(I_2)$ if $I_1\ge cI_2$. 
We use $o(1)$ for quantities that vanish as $m,n,T\to\infty$.

For a multiset $S=\{z_1,\dots,z_N\}$ (in which $z_i=z_j$ is possible for $i\ne j$), its $q$-quantile is $z^*_{\lfloor qN\rfloor}$, where $z_1^*\le\cdots\le z_N^*$ is the non-decreasing rearrangement; if $qN<1$, we define the $q$-quantile as $z_1^*$. We denote the $q$-quantile by $Q_q$.
For simplicity, throughout the theoretical analysis we assume that the relevant proportions of $D$, such as $qD$, are integers. For non-integer values, the rounding corrections are of order $1/D$ and are omitted.

Our analysis uses several binomial random variables and repeatedly relies on the following Chernoff bound.

\begin{lem}[Chernoff bound, see, e.g., {\cite[Sec.~4.7]{ash2012information}}]
\label{thm: chernoff}
Let $X\sim \Bin{N}{q}$. For $k \leq Nq$,
\[
    \mathbb{P}\left(X \leq k\right)
    \leq
    \exp\left(-N \cdot \D\left(\frac{k}{N} \| q\right)\right).
\]
Similarly, for $k \geq Nq$,
\[
    \mathbb{P}(X \geq k)
    \leq
    \exp\left(-N \cdot \D\left(\frac{k}{N} \| q\right)\right).
\]
\end{lem}

The remainder of this section introduces the setup and notation used in the analysis.
We first formalize the streaming measurement model and the streaming QRK algorithm.
We then specify the Massart and oblivious noise models, and finally introduce the normalized clean residual $Z$ together with the associated quantile and truncated-moment quantities used throughout \cref{sec: streaming setting}.

\subsection{Streaming QRK}
\label{sec:models}

The streaming setting, also known as the online setting, has been widely adopted in the analysis of iterative algorithms; see, e.g., \cite{haddock_quantilebased_2022,jeong2025stochastic,das2024near,diakonikolas2022streaming,pesme2020online}. The main difference from the traditional ensemble setting is that fresh measurement vectors are sampled at each iteration independently of the current iterate. For instance, when solving a linear system $\bA\target=\bb$, a single RK step in the classical ensemble setting draws a row from the fixed ensemble $(\bA,\bb)$, whereas in the streaming setting a fresh measurement vector is received at each iteration. Besides reducing the memory load, this independence also enables a sharper analysis through direct expectation calculations.
Following \cite{cai2025subsamplesizequantilebasedrandomized}, we focus on i.i.d. measurement vectors uniformly distributed over the unit sphere $S^{n-1}$.

\begin{assumption}[Measurement vectors]
\label{assump:gaussianai}
Let $n\geq 2$. The measurement vectors $\{\ba_i\}_{i=1}^\infty$ are i.i.d. uniformly distributed over $S^{n-1}$, written as $\ba_i \sim {\rm Unif}(S^{n-1})$.
\end{assumption}

We next describe streaming QRK. Each QRK update relies on a quantile computed from a subsample of size $D$. Consequently, each iteration requires a batch of $D+1$ fresh measurements: $D$ measurements are used to compute the subsampled quantile, and one additional measurement is used for the update. 
We write the corrupted observation associated with a measurement vector $\ba_{\updateid[j]{k+1}}$ as
\[
    b_{k+1,j}
    =
    \langle \ba_{\updateid[j]{k+1}},\target\rangle
    +
    \cv_{\updateid[j]{k+1}}\epsilon_{k+1,j},
\]
where $\cv_{\updateid[j]{k+1}}$ is the corruption indicator and $\epsilon_{k+1,j}$ is the corruption value. 
The same streaming QRK procedure is used under both corruption models considered in this paper; only the assumptions on the corruption values $\epsilon_{k+1,j}$ differ.
We formalize this procedure in \cref{alg:streaming-qrk}.
At each iteration $k$, the $j=0$ measurement is reserved for the update, while $j=1,\dots,D$ are used to compute the empirical quantile $Q_{q,k+1}$ in \eqref{def:subsample quantile}. 
In \eqref{eq:update_step_streaming}, the update sample contributes only if its absolute residual does not exceed the threshold $Q_{q,k+1}$.

\begin{algorithm}[H]
\caption{Streaming Quantile Randomized Kaczmarz (QRK)}
\label{alg:streaming-qrk}
\begin{algorithmic}[1]
\REQUIRE Initial iterate $\bx_0 \in \mathbb{R}^n$, quantile level $q \in (0,1)$, subsample size $D \geq 1$, number of iterations $T$.
\FOR{$k = 0, 1, \dots, T-1$}
    \FOR{$j = 0, 1, \dots, D$}
        \STATE $b_{k+1,j} = \langle \ba_{\updateid[j]{k+1}}, \target \rangle + \cv_{\updateid[j]{k+1}}\,\epsilon_{k+1,j}$
    \ENDFOR
    \STATE Compute the empirical $q$-quantile of absolute residuals over $j=1,\dots,D$:
    \begin{equation}
        Q_{q,k+1}
        :=
        \mathrm{Quantile}_q\!\left(
        \bigl\{
        |\langle \ba_{\updateid[j]{k+1}}, \bx_k \rangle - b_{k+1,j}|
        \bigr\}_{j=1}^D
        \right).
        \label{def:subsample quantile}
    \end{equation}
    \STATE Update using the $j=0$ measurement:
    \begin{equation}
        \bx_{k+1}
        =
        \bx_k
        -
        \bigl(
        \langle \ba_{\updateid{k+1}}, \bx_k \rangle - b_{k+1,0}
        \bigr)
        \cdot
        \mathbf{1}_{|\langle \ba_{\updateid{k+1}}, \bx_k \rangle - b_{k+1,0}| \leq Q_{q,k+1}}
        \cdot
        \ba_{\updateid{k+1}}.
        \label{eq:update_step_streaming}
    \end{equation}
\ENDFOR
\end{algorithmic}
\end{algorithm}

\subsection{Corruption models} \label{sec:corruption models}
We now specify the corruption models used in \cref{alg:streaming-qrk}. In the classical ensemble setting, QRK is capable of exactly solving sparsely corrupted linear systems in which $\beta m$ entries of the clean observation vector $\bb$ are replaced by arbitrary values, referred to as a $\beta$-fraction of adversarial corruption. Modeling adversarial corruption in the streaming setting is subtler because there is no fixed ensemble of measurements; see \cite[Sec. 1.2]{jeong2025stochastic}. 
In this paper, we consider two corruption models that are widely studied in the literature: the oblivious noise model and the Massart noise model.

The weaker of the two is the {\it oblivious noise model} \cite{pesme2020online}, in which each measurement is independently corrupted with probability $\beta$ by additive noise that is independent of all other randomness.
Our main focus is on the stronger {\it semi-random Massart noise model} \cite{diakonikolas2021relu,jeong2025stochastic}. Under Massart noise, measurements are independently corrupted with probability $\beta$, but once a measurement is corrupted, its corruption value may be chosen adversarially. In other words, the adversary cannot choose which measurements to corrupt, but can choose the corruption values using any available information, such as the true solution and the measurement vector. 
The following definition formalizes both models, where we omit the step index $k+1$ for brevity.

\begin{definition}[Sparse corruption models]
\label{def:sparse-corruption}
The corruption of the $j$-th measurement is $\cv_j\epsilon_j$, where the indicators $(\cv_j)_{j=0}^D$ are i.i.d. $\operatorname{Bernoulli}(\beta)$ variables independent of the measurement vectors and the previous iterations, and the corruption values $(\epsilon_j)_{j=0}^D \in \mathbb{R}^{D+1}$ are arbitrary. The only difference between Massart noise and oblivious noise is whether the corruption values can be chosen adversarially:
\begin{itemize}
    \item Under Massart noise, the corruption values $(\epsilon_j)_{j=0}^D$ may be chosen jointly and adversarially using any available information, including $(\ba_j)_{j=0}^D$, $\target$, $\bx_k$, and $(\cv_j)_{j=0}^D$.
    \item Under oblivious noise, the corruption values are chosen without using this information; they may be fixed in advance or drawn independently of each other and of all other randomness. We treat fixed values as degenerate random variables and do not distinguish the two cases below.
\end{itemize}
\end{definition}

With the streaming algorithm and corruption models in place, we next record the normalized clean residual used throughout the analysis.
We write $\bX_k$ for the random iterate after $k$ steps and $\bx_k$ for a fixed realization; the corresponding error is $\be_k:=\bx_k-\target$.
For $j=0,\dots,D$, the residuals satisfy
\[
    \langle \ba_{\updateid[j]{k+1}}, \bx_k \rangle - b_{k+1,j}
    =
    \langle \ba_{\updateid[j]{k+1}}, \be_k \rangle
    -
    \cv_{\updateid[j]{k+1}}\,\epsilon_{k+1,j}
    =
    Z\frac{\|\be_k\|}{\sqrt{n}}
    -
    \cv_{\updateid[j]{k+1}}\,\epsilon_{k+1,j},
\]
where
\begin{equation}
    Z := \langle \sqrt{n}\,\ba, \mathbf{e}^{(1)} \rangle,
    \label{eq:Z_def}
\end{equation}
with $\ba \sim \mathrm{Unif}(S^{n-1})$ for some $n \geq 2$ and $\mathbf{e}^{(1)}$ the first canonical basis vector.
This representation follows from the independence of the fresh measurement vectors from the current iterate and the rotational invariance of the unit sphere.
For large $n$, $Z$ is approximately $\mathcal{N}(0,1)$ \cite{vershynin2018high}; we use this approximation to obtain explicit numerical values of the relevant quantiles and expectations in \cref{sec: explicit D}.

Two quantities built from $Z$ appear repeatedly below.
We write $\Phi_\alpha$ for the $\alpha$-quantile of $|Z|$, i.e.\ $\mathbb{P}(|Z|\le \Phi_\alpha)=\alpha$, and define the truncated second moment
\[
    g(t):=\mathbb{E}\bigl[Z^2\mathbf{1}_{|Z|\le t}\bigr],
    \qquad t\ge 0.
\]

\section{Main Results}
\label{sec: streaming setting}
This section develops the main theoretical guarantees for streaming QRK.
We begin with high-probability two-sided bounds on the subsampled quantile $Q_{q,k+1}$ (\cref{lem: sub quantile streaming - simple}), which apply under both Massart and oblivious noise.
Using these bounds, we analyze the Massart noise model: one-step corrupted and uncorrupted updates yield linear convergence (\cref{thm: convergence in expectation-streaming}), after which we characterize the maximum tolerable corruption level $\beta^*$ and the required subsample size $D^*$ (\cref{prop: maximum corruption level,prop: explicit subsample size}).
Finally, under oblivious noise we obtain a sharper joint one-step analysis and the corresponding improvements for $\beta^*_{\mathrm{obl}}$ and $D^*_{\mathrm{obl}}$.
 
Compared to QRK with full-sample quantiles \cite{haddock_quantilebased_2022,steinerberger2023quantile}, one additional challenge in analyzing the subsampled variant where $D\ll m$ \cite{cai2025subsamplesizequantilebasedrandomized} is to control the fluctuating quantile. 
In our streaming setting, we establish the following two-sided probabilistic bounds on the subsampled quantile $Q_{q,k+1}$ defined in \cref{def:subsample quantile}.

\begin{lem} \label{lem: sub quantile streaming - simple} 
     Let $\bx_k\in \mathbb{R}^n$, $q \in (0,1)$, $\epsl \in (0,q)$ and $\epsu \in (0,1-q)$ be given constants, and $\beta \in (0, \min\{q - \epsl, 1-q - \epsu\})$ be the corruption level. Over the randomness of $(\ba_{k+1,j}, \cv_{k+1,j})_{j=1}^D$, we have 
     \begin{gather}
         \mathbb{P}\bigg(Q_{q,k+1} \leq \Phi_{\ur} \frac{\|\be_k\|}{\sqrt{n}}\bigg) \ge 1- \exp(- \D(q \| q + \epsu) D), \label{eq: quantile upper}\\
         \mathbb{P}\bigg( Q_{q,k+1}\ge \Phi_{\lr} \frac{\|\be_k\|}{\sqrt{n}} \bigg) \ge 1- \exp(- \D(1 - q \| 1 - q + \epsl) D), \label{eq: quantile lower}
     \end{gather}
    where $\Phi$ is the quantile function of $|Z|$ such that
    $
        \mathbb{P}(|Z|\leq\Phi_{\alpha}) = \alpha
  $ 
  with $Z$ defined in \cref{eq:Z_def}.
\end{lem}

\begin{proof}
    We use the convention from 
    \cref{sec:prelim} that quantities such as $qD$ are integers.
    Recall that given $\bx_k$, $Q_{q,k+1}$ was defined in \cref{def:subsample quantile}.
    We first consider the uncorrupted residuals,
    $$
        \left\{\left|\left\langle\mathbf{a}_{k+1,j}, \be_k\right\rangle\right|: j \in \{1, \dots, D\}\right\}.
    $$
    In the streaming setting, for all $j \in \subsample{i}{k+1}$, $\bm{a}_{k+1,j}$ is independent of $\bX_k$. As a result, the random variables $\left|\left\langle\sqrt{n}\mathbf{a}_{k+1,j}, \be_k \right\rangle\right|$ are i.i.d. with the same distribution as $\|\be_k\||Z|$; see \cref{eq:Z_def}. We define
    \[
    Z_j := \mathbf{1}\bigg(\left|\left\langle\ba_{k+1,j}, \be_k\right\rangle\right| \leq \Phi_{\ur} \frac{\|\be_k\|}{\sqrt{n}},\; \cv_{k+1,j}=0\bigg),
    \]
    where $\mathbf{1}(\cdot)$ is the indicator function. Then we have
    $$
            \mathbb{P}(Z_j = 1) = \mathbb{P}(\cv_{k+1,j}=0 \cap \left|\left\langle\ba_{k+1,j}, \be_k\right\rangle\right| \leq \Phi_{\ur} \frac{\|\be_k\|}{\sqrt{n}}) = (1 - \beta)\ur = q + \epsu
    $$
    from the definition of $\Phi_{\alpha}$ and the independence of $\cv_{k+1,j}$ from $\ba_{k+1,j}$ and $\bX_k$. 
       Using the Chernoff bound (\cref{thm: chernoff}), we have that
    $$
        \mathbb{P}(\sum_{j=1}^D Z_j \leq  q  D) \leq \exp(- \D(q \| q + \epsu) D).
    $$
    On the complement of this event, more than $qD$ of the subsampled residuals are no larger than $\Phi_{\ur}\|\be_k\|/\sqrt n$.
    Therefore, 
    $$
        Q_{q,k+1} \leq \Phi_{\ur} \frac{\|\be_k\|}{\sqrt{n}}
    $$
    holds with probability at least $1-\exp(- \D(q \| q + \epsu) D)$.

    We then define
    \[
    \widetilde{Z}_j := \mathbf{1}\bigg(\left|\left\langle\ba_{k+1,j}, \be_k\right\rangle\right| \geq \Phi_{\lr} \frac{\|\be_k\|}{\sqrt{n}},\; \cv_{k+1,j}=0\bigg),
    \]
    where $\mathbf{1}(\cdot)$ is the indicator function. Then we have
    $$
        \begin{aligned}
            \mathbb{P}(\widetilde{Z}_j = 1) =\mathbb{P}(\cv_{k+1,j}=0 \cap \left|\left\langle\ba_{k+1,j}, \be_k\right\rangle\right| \geq \Phi_{\lr} \frac{\|\be_k\|}{\sqrt{n}})
            = (1 - \beta)(1 - \lr) = 1 - q + \epsl,
        \end{aligned}
    $$
    similarly. 
    Using the Chernoff bound,
    we know
    $$
            \mathbb{P}(\sum_{j=1}^D \widetilde{Z}_j \leq  (1 - q )  D) \leq \exp(- \D(1 - q \| 1 - q + \epsl) D).
    $$
    On the complement of this event, more than $(1-q)D$ of the subsampled residuals are no smaller than
$\Phi_{\lr}\|\be_k\|/\sqrt n$. Therefore,
    $$
        Q_{q,k+1} \geq \Phi_{\lr} \frac{\|\be_k\|}{\sqrt{n}}
    $$
   holds with probability at least $1 - \exp(- \D(1 - q \| 1 - q + \epsl) D)$. 
\end{proof}
 
Note that although $Q_{q,k+1}$ depends on the corruption values $(\epsilon_{k+1,j})_{j=1}^D$, the probability bounds here are over the randomness of $(\ba_{k+1,j}, \cv_{k+1,j})_{j=1}^D$ and hold uniformly for any adversarial corruption values in the Massart noise model (and thus also under the weaker oblivious noise model). 
For example, \cref{eq: quantile upper} can be explicitly written as
$
\mathbb{P}\bigg(\sup_{(\epsilon_{k+1,j})_{j=1}^D} Q_{q,k+1} \leq \Phi_{\ur} \frac{\|\be_k\|}{\sqrt{n}}\bigg) \ge 1- \exp(- \D(q \| q + \epsu) D);
$
and \cref{eq: quantile lower} can be written as
$
\mathbb{P}\bigg(\inf_{(\epsilon_{k+1,j})_{j=1}^D} Q_{q,k+1} \geq \Phi_{\lr} \frac{\|\be_k\|}{\sqrt{n}}\bigg) \ge 1- \exp(- \D(1 - q \| 1 - q + \epsl) D).
$

\subsection{Massart Noise Model}

We now use the two-sided bounds in \cref{lem: sub quantile streaming - simple} to analyze the Massart noise model. 
The upper bound on $Q_{q,k+1}$ controls the one-step error increase when a corrupted row is accepted (\cref{lem: corrupted increase}), while the lower bound ensures sufficient decrease from accepted uncorrupted rows (\cref{lem: error decreased}). Combining these two one-step bounds yields the convergence result in \cref{thm: convergence in expectation-streaming}. 
Under the Massart noise model, probabilities and expectations are taken with respect to the randomness of the relevant sampling rows and corruption indicators $(\ba_{\ell,j}, \cv_{\ell,j})_{j=0}^D$, where $\ell$ indexes the relevant iterations. The resulting bounds hold uniformly over the adversarial corruption values $(\epsilon_{\ell,j})_{j=0}^D$.
For simplicity, within the Massart analysis, we omit the corresponding supremum or infimum notation when the uniform interpretation is clear from context.

\subsubsection{One-Step Bounds and Convergence}
In this subsection, we establish one-step bounds given the current iterate $\bx_k$.
We first bound the error increase after a corrupted update. 
To exclude the worst-case effect of the corruption, we define a successful event $S_{k+1}$ where either the quantile is well-bounded or the update is uncorrupted:
\begin{equation} \label{eq: successful event}
    S_{k+1} = \{ \sup_{(\epsilon_{k+1,j})_{j=1}^D} Q_{q,k+1} \leq \tilde{m} \frac{\|\be_k\|}{\sqrt{n}} \quad \text{or} \quad  \cv_{k+1,0} = 0\},\quad k = 0, 1, \ldots, T-1.
\end{equation} 
Here, $S_{k+1}$ is defined given $\bx_k$ and a threshold $\tilde{m}>0$, where $Q_{q,k+1}$ is the subsampled quantile associated with $\be_k = \bx_k - \target$. The randomness of $S_{k+1}$ is over the sampling rows and corruption indicators $(\ba_{k+1,j}, \cv_{k+1,j})_{j=0}^D$.
The following lemma establishes that, restricted to this successful event, the error increase resulting from a corrupted update row can be explicitly bounded.

\begin{lem}[Corrupted Update] \label{lem: corrupted increase}
    Given $\bx_k$, we have
\begin{align} \label{eq: corrupted update increase}
\mathbb{E}\left[\left\|\boldsymbol{X}_{k+1}-\target\right\|^2 \mathbf{1}_{S_{k+1}} \mid \boldsymbol{X}_k=\boldsymbol{x}_k, \cv_{k+1,0} = 1\right]  \leq  \|\be_k\|^2 + f(\tilde{m}) \frac{\|\be_k\|^2}{n},
\end{align}
where
$
    f(\tilde{m}) = \tilde{m}^2 + 2 \tilde{m} \mathbb{E}|Z|,
$
and $Z$ is distributed as in \cref{eq:Z_def}.
\end{lem}

\begin{proof}
Let $\epsilon=\epsilon_{k+1,0}$. When $\cv_{k+1,0}=1$, the update step size in \cref{eq:update_step_streaming} is
\[
s = \bigl(\langle \ba_{\updateid{k+1}}, \bx_k \rangle - b_{k+1,0}\bigr)
\mathbf{1}_{|\langle \ba_{\updateid{k+1}}, \bx_k \rangle - b_{k+1,0}| \leq Q_{q,k+1}}
= (\langle \bm{a}_{\updateid{k+1}}, \be_k \rangle - \epsilon)
\mathbf{1}_{|\langle \bm{a}_{\updateid{k+1}}, \be_k \rangle - \epsilon| \leq Q_{q,k+1}} .
\]
Then
\begin{align*}
&\mathbb{E}\left[\left\|\boldsymbol{X}_{k+1}-\target\right\|^2 \mathbf{1}_{S_{k+1}} \mid \boldsymbol{X}_k=\boldsymbol{x}_k, \cv_{k+1,0} = 1\right] \\
&= \mathbb{E}\left[\left\| \be_k - s \bm{a}_{\updateid{k+1}} \right\|^2 \mathbf{1}_{S_{k+1}} \mid \boldsymbol{X}_k=\boldsymbol{x}_k, \cv_{k+1,0} = 1\right] \\ 
& \stackrel{(a)}{=} \mathbb{E}\left[\|\be_k\|^2 \mathbf{1}_{S_{k+1}} - 2s\langle \bm{a}_{\updateid{k+1}}, \be_k \rangle \mathbf{1}_{S_{k+1}} + s^2 \mathbf{1}_{S_{k+1}}   \mid \boldsymbol{X}_k=\boldsymbol{x}_k,  \cv_{k+1,0} = 1\right]   
\end{align*}
where (a) 
uses the fact that $\|\ba_{k+1,0}\|=1$. From the threshold indicator in $s$, we have
\[-2s\langle \ba_{k+1,0},\be_k\rangle +s^2 \le 2Q_{q,k+1}|\langle\ba_{k+1,0},\be_k\rangle|+Q_{q,k+1}^2,\]
and this bound is tight: the adversary can choose $\epsilon$ so that $\langle \bm{a}_{\updateid{k+1}}, \be_k \rangle - \epsilon = -Q_{q,k+1} \operatorname{sign}(\langle \bm{a}_{\updateid{k+1}}, \be_k \rangle)$. 
Under the conditioning $\cv_{k+1,0}=1$, the event $S_{k+1}$ reduces to
$\{Q_{q,k+1}\le \tilde{m}\|\be_k\|/\sqrt n\}$, and we have
\begin{align*}
&\mathbb{E}\left[- 2s\langle \bm{a}_{\updateid{k+1}}, \be_k \rangle \mathbf{1}_{S_{k+1}} + s^2 \mathbf{1}_{S_{k+1}} \mid \boldsymbol{X}_k=\boldsymbol{x}_k, \cv_{k+1,0} = 1\right] \\
& \leq \mathbb{E}\left[2 Q_{q,k+1} |\langle \bm{a}_{\updateid{k+1}}, \be_k\rangle| \mathbf{1}_{S_{k+1}} + Q_{q,k+1}^2 \mathbf{1}_{S_{k+1}} \mid \boldsymbol{X}_k=\boldsymbol{x}_k, \cv_{k+1,0} = 1\right] \\
&\leq  2 \left(\tilde{m} \frac{\|\be_k\|}{\sqrt{n}}\right) \mathbb{E} \left[ \left|\left\langle \bm{a}_{\updateid{k+1}}, \be_k\right\rangle\right| \mid \bX_k = \bx_k \right] + \left(\tilde{m} \frac{\|\be_k\|}{\sqrt{n}}\right)^2, 
\end{align*}
where the last step uses the upper bound on $Q_{q,k+1}$ on $S_{k+1}$ and the independence of $\cv_{k+1,0}$.
Since $\ba_{k+1,0}$ is independent of $\bX_k$ and is rotationally invariant, we have
$
\mathbb{E}\left[\left|\left\langle \sqrt{n}\bm{a}_{\updateid{k+1}}, \be_k\right\rangle\right|\right] =  \|\be_k\|\mathbb{E}|Z|.
$
Combining this identity with $\mathbb{E}[\|\be_k\|^2\mathbf{1}_{S_{k+1}}\mid \bX_k=\bx_k,\cv_{k+1,0}=1]\leq \|\be_k\|^2$, we obtain $f(\tilde{m}) = \tilde{m}^2 + 2\tilde{m}\,\mathbb{E}|Z|$, thereby completing the proof.

\end{proof}

Next, we show that an uncorrupted update yields an expected error decrease whenever the quantile is sufficiently large.
Define the corresponding successful event by
\begin{equation} \label{eq: S_{k+1}^*}
    S_{k+1}^*
    :=
    \left\{
    \inf_{(\epsilon_{k+1,j})_{j=1}^D}
    Q_{q,k+1}
    \ge
    \Phi_{\alpha_0}\frac{\|\be_k\|}{\sqrt{n}}
    \right\}.
\end{equation}
Here, the probability is taken over the sampling rows and corruption indicators
$(\ba_{k+1,j},\cv_{k+1,j})_{j=1}^D$.

\begin{lem}[Uncorrupted Update]  \label{lem: error decreased}
    
    Given $\bx_k$, let $\alpha_0 \in [0,1]$, define the event $S_{k+1}^*$ as in \cref{eq: S_{k+1}^*}, and suppose that $\mathbb{P}\left((S_{k+1}^*)^c \right) \leq p_l$ for some $p_l \in [0,1]$.
Then,
\begin{align} \label{eq: uncorrupted update decrease}
\mathbb{E}\left[\left\|\boldsymbol{X}_{k+1}-\target\right\|^2  \mid \boldsymbol{X}_k=\boldsymbol{x}_k, \cv_{k+1,0} = 0\right]  \leq  \|\be_k\|^2 - (1 - p_l) g(\Phi_{\alpha_0}) \frac{\|\be_k\|^2}{n},
\end{align}
where $g(\Phi_{\alpha_0}) = \truncZsq$, with $Z$ distributed as in \cref{eq:Z_def}.
\end{lem}

\begin{proof}
Condition on $\bX_k=\bx_k$ and $\cv_{k+1,0}=0$. The update may be accepted, in which case
$\|\be_{k+1}\|^2 = \|\be_k\|^2 - \langle \ba_{\updateid{k+1}}, \be_k\rangle^2$, or rejected, in which case $\|\be_{k+1}\|^2 = \|\be_k\|^2$. In both cases, the error is non-expansive, i.e., $\|\be_{k+1}\|^2 \leq \|\be_k\|^2$. Furthermore, on the event $S_{k+1}^*$, we have $Q_{q,k+1} \ge \Phi_{\alpha_0}\frac{\|\be_k\|}{\sqrt{n}}$, so every uncorrupted update with
$
\left|\left\langle \ba_{\updateid{k+1}}, \be_k\right\rangle\right| \leq \Phi_{\alpha_0} \frac{\|\be_k\|}{\sqrt{n}}
$
will be accepted, where $\|\be_{k+1}\|^2 = \|\be_k\|^2 - \langle \ba_{\updateid{k+1}}, \be_k\rangle^2$ is guaranteed.
Let
$
A_{k+1}:=\left\{|\langle \ba_{\updateid{k+1}}, \be_k\rangle|
\leq \Phi_{\alpha_0}\frac{\|\be_k\|}{\sqrt{n}}\right\}.
$
Then we have
    \begin{align*}
&\mathbb{E}\left[\left\|\boldsymbol{X}_{k+1}-\target\right\|^2  \mid \boldsymbol{X}_k=\boldsymbol{x}_k, \cv_{k+1,0} = 0\right] \leq \|\be_k\|^2
- \mathbb{E}\left[\langle \ba_{\updateid{k+1}}, \be_k\rangle^2 \,
\mathbf{1}_{A_{k+1}}\mathbf{1}_{S_{k+1}^*}
\mid \boldsymbol{X}_k=\boldsymbol{x}_k\right], 
    \end{align*}
where we use the independence of $\cv_{k+1,0}$.
By rotational invariance,
$\langle\ba_{\updateid{k+1}},\be_k\rangle^2$ has the same distribution as
$Z^2\|\be_k\|^2/n$, where $Z$ is distributed as in \cref{eq:Z_def}. Since the update row is independent of the quantile subsample determining $S_{k+1}^*$, we have
\begin{align*}
\mathbb{E}\left[\langle \ba_{\updateid{k+1}}, \be_k\rangle^2
\mathbf{1}_{A_{k+1}}\mathbf{1}_{S_{k+1}^*}
\mid \boldsymbol{X}_k=\boldsymbol{x}_k\right] 
\geq (1-p_l)\frac{\|\be_k\|^2}{n}
\mathbb{E}\left[Z^2\mathbf{1}_{|Z|\leq\Phi_{\alpha_0}}\right],
\end{align*}
which yields \cref{eq: uncorrupted update decrease}.
\end{proof}

 Combining the one-step bounds in \cref{lem: corrupted increase} and \cref{lem: error decreased}, we obtain the following convergence result for the streaming setting under Massart noise. 

\newcommand{\mQ}{\Phi_{1 - \urr}\frac{\|\be_k\|}{\sqrt{n}}}
\begin{thm} \label{thm: convergence in expectation-streaming}
      Let $T$ be the number of iterations, $D$ be the subsample size, $q \in (0,1)$ be the quantile parameter, and $\beta \in (0,1)$ be 
    the corruption level. Choose $\alpha,\alpha'>0$ such that 
\begin{gather} 
         \label{eq: valid conditions beta}
                 \epsl =  q - \alpha - \beta >0, \quad  \epsu = 1 - q - (\beta  + \alpha' ) > 0, \quad p_l = \exp(- \D(q \| \beta + \alpha) D),\\
                  \cadv := (1-\beta) (1 - p_l) g(\Phi_{\lrr}) - \beta f(\Phi_{1 - \urr}) > 0, \label{eq: positive contraction}
             \end{gather}
            where $\Phi$ is as in \cref{lem: sub quantile streaming - simple}, $f$ is the error-increase function defined in \cref{lem: corrupted increase}, and $g$ is the error-decrease function defined in \cref{lem: error decreased}. Let $\tau_2$ be the stopping time
       \begin{equation} \label{eq:def_tau2 3}
    \tau_2 =\inf \left\{t \geq 1: S_t^c\right\},
   \end{equation}
   where $S_t$ is the successful event defined in \cref{eq: successful event} evaluated at the iterate
   $\bm{X}_{t-1}$, with $\tilde{m} = \Phi_{1 - \urr}$.
Then
 \begin{gather} 
 \mathbb{E}\left(\left\|\bX_T - \target\right\|^2 \mathbf{1}_{\tau_2 > T}\right) \leq\left(1-\frac{ \cadv }{n}\right)^T\left\|\init - \target\right\|^2,
 \nn\\\label{eq:failure-prob}
    \mathbb{P}(\tau_2 \leq T) \leq 1 - (1 - \beta \exp ( - \D(1 - q \| \beta + \alpha') \cdot D) )^T \leq T \beta \exp ( - \D(1 - q \| \beta + \alpha') \cdot D). 
    \end{gather}   
\end{thm}
\begin{proof}
    Recall that
    $
    S_{k+1} = \left\{ Q_{q,k+1} \leq \Phi_{1 - \urr} \frac{\|\bX_k - \target\|}{\sqrt{n}} \quad \text{or} \quad \cv_{k+1,0} = 0\right\}.
    $
    We first upper bound its failure probability using \cref{lem: sub quantile streaming - simple}. 
    Given any fixed $\bx_k$, the event
    $
        Q_{q,k+1} > \Phi_{\frac{q + \epsu}{1 - \beta}} \frac{\|\be_k\|}{\sqrt{n}} = \Phi_{1 - \urr} \frac{\|\be_k\|}{\sqrt{n}}
    $
    occurs with probability at most $\exp ( - \D( q \| q + \epsu) \cdot D) = \exp ( - \D(1 - q \| \beta + \alpha') \cdot D)$, and the event $\cv_{k+1,0} = 1$ holds with probability $\beta$ independently.
    Therefore, we have 
    $$
    \mathbb{P}((S_{k+1})^c \mid \bm{X}_{k} = \bx_k) \leq \beta \exp ( - \D(1 - q \| \beta + \alpha') \cdot D).
    $$
Since, conditional on $\boldsymbol{X}_k=\bx_k$, the remaining randomness in $S_{k+1}$ is independent of the previous iterations, we have
    $$
    \mathbb{P}\left(\tau_2>k+1 \mid \tau_2>k\right) \geq 1-\beta \exp \left(-\D\left(1-q \| \beta+\alpha^{\prime}\right) D\right).
    $$
    Iterating this bound gives
    \begin{align*}
        \mathbb{P}(\tau_2 \leq T)
        &= 1- \mathbb{P}(\tau_2 > T) \\
        &\leq 1 - \left(1 - \beta \exp ( - \D(1 - q \| \beta + \alpha') \cdot D) \right)^T \\
        &\leq T \beta \exp ( - \D(1 - q \| \beta + \alpha') \cdot D),
    \end{align*}
    where the last inequality follows from Bernoulli's inequality, $(1-x)^T \geq 1 - Tx$, with $x  = \beta \exp ( - \D(1 - q \| \beta + \alpha') \cdot D) \in [0,1]$ and $T \geq 1$.
    Following \cite[Lemma 3.5]{cai2025subsamplesizequantilebasedrandomized}, we estimate the one-step contraction with the indicator of $S_{k+1}$:
    $$
        \begin{aligned}
             & \mathbb{E}\left[\left\|\boldsymbol{X}_{k+1}-\target\right\|^2 \mathbf{1}_{S_{k+1}} \mid \boldsymbol{X}_k=\boldsymbol{x}_k\right]                                                                    \\
             & = \beta \mathbb{E}\left[\left\|\boldsymbol{X}_{k+1}-\target\right\|^2 \mathbf{1}_{S_{k+1}} \mid \boldsymbol{X}_k=\boldsymbol{x}_k,  \cv_{k+1,0} = 1\right]      \\
             & + (1 - \beta) \mathbb{E}\left[\left\|\boldsymbol{X}_{k+1}-\target\right\|^2 \mathbf{1}_{S_{k+1}}  \mid \boldsymbol{X}_k=\boldsymbol{x}_k,   \cv_{k+1,0} = 0\right] \\
             & \leq \beta I_1 + (1 - \beta) I_2.
        \end{aligned}
    $$

     \subsubsection*{Case I: Corrupted Update Sample $\cv_{k+1,0} = 1$}
	    By the definition of $S_{k+1}$, we take $\tilde{m} = \Phi_{1 - \urr}$ and apply \cref{lem: corrupted increase}:
    $$
        \begin{aligned}
            I_1 & = \mathbb{E}\left[\left\|\boldsymbol{X}_{k+1}-\target\right\|^2 \mathbf{1}_{S_{k+1}} \mid \boldsymbol{X}_k=\boldsymbol{x}_k,  \cv_{k+1,0} = 1\right]                                                                                                             \\
            &\leq  \|\be_k\|^2 + f(\Phi_{1 - \urr}) \frac{\|\be_k\|^2}{n}.
        \end{aligned}
    $$

     \subsubsection*{Case II: Uncorrupted Update Sample $\cv_{k+1,0} = 0$}
	     By \cref{lem: sub quantile streaming - simple}, we have $Q_{q,k+1} \ge \Phi_{\lr}\frac{\|\be_k\|}{\sqrt{n}} = \Phi_{\lrr} \frac{\|\be_k\|}{\sqrt{n}}$ with probability at least $1 - p_l = 1 - \exp(-\D(q \| q - \epsl) D) = 1 - \exp(-\D(q \| \beta + \alpha) D)$. Taking $\alpha_0 = \lrr$ and this value of $p_l$, we apply \cref{lem: error decreased}:
    $$
        \begin{aligned}
           \mathbb{E}\left[\left\|\boldsymbol{X}_{k+1}-\target\right\|^2 \mathbf{1}_{S_{k+1}}  \mid \boldsymbol{X}_k=\boldsymbol{x}_k,  \cv_{k+1,0} = 0 \right]  
            & \leq  \mathbb{E}\left[\left\|\boldsymbol{X}_{k+1}-\target\right\|^2  \mid \boldsymbol{X}_k=\boldsymbol{x}_k,  \cv_{k+1,0} = 0 \right] \\
            &\leq \|\be_k\|^2 - (1 - p_l) g(\Phi_{\lrr}) \frac{\|\be_k\|^2}{n}.\\ 
        \end{aligned}
    $$

    Combining the two cases gives
        \begin{align}
             & \mathbb{E}\left[\left\|\boldsymbol{X}_{k+1}-\target\right\|^2 \mathbf{1}_{S_{k+1}} \mid \boldsymbol{X}_k=\boldsymbol{x}_k\right] \nonumber \\
             & \leq  \|\be_k\|^2  - \left((1-\beta) (1 - p_l) g(\Phi_{\lrr}) - \beta f(\Phi_{1 - \urr})\right)\frac{\|\be_k\|^2}{n}  \nonumber\\
             & = \left(1 - \frac{\cadv}{n} \right) \|\be_k\|^2, \label{eq: adv comibine error}
        \end{align}
    where $\cadv := (1-\beta) (1 - p_l) g(\Phi_{\lrr}) - \beta f(\Phi_{1 - \urr}) > 0$ by assumption. 
    Since the one-step contraction holds for every fixed $\bx_k$ and $k=0,1,2,\ldots,T-1$, the induction argument in \cite[Lemma 3.5]{cai2025subsamplesizequantilebasedrandomized} gives
    \begin{align*} 
        \mathbb{E}\left(\left\|\bm{X}_T - \target\right\|^2 \mathbf{1}_{\tau_2>T}\right) \le \Big(1-\frac{\cadv}{n}\Big)^T \mathbb{E}\left(\left\|\bm{X}_0 - \target\right\|^2\right),
	      \end{align*}
	    which completes the proof.
\end{proof}

By Markov's inequality, \cref{thm: convergence in expectation-streaming} yields the following high-probability guarantee.

\begin{corollary} \label{cor: high probability bound}
    In the setting of \cref{thm: convergence in expectation-streaming}, let $T\ge 1$ and $\bx_0 \neq \target$. If $\mathbb{P}(\tau_2 \leq T) \leq \delta_f < \frac{1}{2}$, then
   $\|\bm{X}_T -\target\|^2 \le (1-\frac{ \cadv }{2n})^T\|\bm{x}_0-\target\|^2$ with probability at least $1 - \delta_f - 2\exp(- \frac{\cadv T}{2n})$. 
\end{corollary}
\begin{proof}
  Let $\Omega=\{\tau_2 \leq T\}$. Since $\mathbb{P}(\Omega)\leq\delta_f<1/2$, \cref{thm: convergence in expectation-streaming} gives
    \begin{align*}
        \mathbb{E}\big(\|\bm{X}_T-\target\|^2 \mid \Omega^c\big)
        \leq 2\mathbb{E}\big(\|\bm{X}_T-\target\|^2 \mathbf{1}_{\Omega^c}\big)
        \leq 2 \Big(1-\frac{\cadv}{n}\Big)^T \|\bm{x}_0-\target\|^2.
    \end{align*}
    Thus, for any $\epsilon>0$, Markov's inequality implies
    \begin{align*}
        \mathbb{P}\Big(\|\bm{X}_T-\target\|^2 \ge \epsilon \|\bm{x}_0-\target\|^2\Big)
        &\leq
        \mathbb{P}\Big(\|\bm{X}_T-\target\|^2 \ge \epsilon \|\bm{x}_0-\target\|^2 \mid \Omega^c\Big)
        + \mathbb{P}(\Omega) \\
        &\leq \frac{2(1-\frac{\cadv}{n})^T}{\epsilon} + \delta_f.
    \end{align*}
    Taking $\epsilon = (1 - \frac{\cadv}{2n})^T$ and using $\ln(1-x)\leq -x$, we obtain
    \[
    \frac{2(1 - \cadv/n)^T}{\epsilon} \leq 2 \left( \frac{1 - \cadv/n}{1 - \cadv/(2n)} \right)^T \leq 2 \Big(1 - \frac{\cadv}{2n}\Big)^T \leq 2 \exp\Big(-\frac{\cadv T}{2n}\Big),
    \]
    which proves the claim.
\end{proof}

\begin{remark}[Technical comparison to previous works]\label{rem:streaming-sharper} Due to the independence between $\bX_k$ and $\bm{a}_{\updateid{k+1}}$, the key quantities in our streaming setting can be computed directly via expectations. This yields sharper constants compared to the ensemble setting analyzed in previous works, which instead control the same terms via spectral properties of the matrix $\bm{A}$ 
\cite{cai2025subsamplesizequantilebasedrandomized, steinerberger2023quantile}. For instance, in the corrupted update analysis (\cref{lem: corrupted increase}), previous works introduced $c_B := \sigma_{\max}(\tilde{\bm{A}})/\sqrt{\beta}$. %
Moreover, for the uncorrupted update analysis (\cref{lem: error decreased}), the role of an important quantity $(\inf_{\substack{S \subseteq[m]:|S| \geq \alpha_0 m}} \sigma_{\min}(\tilde{\mathbf{A}}_S))^2$ in previous works is now played by a simple expectation $\truncZsq$.
Interestingly, what \cite{steinerberger2023quantile} treats as a limiting assumption for the uniform minimum singular value as $m \rightarrow \infty$ emerges naturally in our streaming setting as an exact consequence of expectation computations.

\end{remark}

\subsubsection{Maximum Tolerable Corruption Level and Minimum Subsample Size}

In light of the convergence results (\cref{thm: convergence in expectation-streaming,cor: high probability bound}), two fundamental questions arise: what is the maximum tolerable corruption level $\beta^*$, and, for $\beta<\beta^*$, what is the minimum required subsample size $D^*$? 
Both questions reduce to the feasibility of a triplet $(D,\alpha,\alpha')$ in the domain
\begin{equation}
    D \in \mathbb{Z}_{\ge 1}, \qquad
    \alpha \in (0,q-\beta), \qquad
    \alpha' \in (0,1-q-\beta), \label{eq:alpha_constraints}
\end{equation}
satisfying the two core constraints from \cref{thm: convergence in expectation-streaming}, described below. 

\noindent\textbf{Positive contraction.}
The first core constraint, corresponding to \cref{eq: positive contraction} in \cref{thm: convergence in expectation-streaming}, requires the following one-step contraction rate to be positive:
\begin{equation}\label{eq:cadv_explicit}
\cadv(D,\alpha,\alpha')
:= (1-\beta)(1-p_l(D,\alpha))\,g(\Phi_{\lrr}) - \beta f(\Phi_{1-\urr}) > 0.
\end{equation}
Here, $p_l(D,\alpha)$ bounds the probability that the subsampled quantile is too small, while $g$ and $f$ quantify the respective error decrease and increase:
            \begin{align}
     p_l(D,\alpha) = \exp(-\D(q \| \beta + \alpha) D),\quad  
     g(t) := \mathbb{E}[Z^2\mathbf{1}_{|Z|\le t}],\quad
      f(\tilde{m}) = \tilde{m}^2 + 2\tilde{m}\mathbb{E}|Z|, \label{eq:p_l_constraint_g_f_definition}
            \end{align}
where $Z$ has the distribution specified in \cref{eq:Z_def}, and $\Phi$ denotes the quantile function of $|Z|$. 
Directly from the definitions, both $f$ and $g$ are nonnegative and non-decreasing on $[0, \infty)$. 

\noindent\textbf{Small failure probability.}
The second core constraint, corresponding to \cref{eq:failure-prob} in \cref{thm: convergence in expectation-streaming}, requires the probability of at least one catastrophic update over $T$ iterations to be at most a prescribed $\delta_f\in(0,\frac12)$. This yields the following lower bound on $D$ in terms of $\alpha'$:
\begin{equation}\label{eq:D_fail}
\begin{aligned}
p_{\mathrm{fail}}(D,\alpha') &:= 1 - (1 - \beta \exp ( - \D(1 - q \| \beta + \alpha') \cdot D))^T \le \delta_f  \\
&\Leftrightarrow D \ge D_{\mathrm{fail}}(\alpha')
:= \frac{-\log\bigl(\frac{1-(1-\delta_f)^{1/T}}{\beta}\bigr)}{\D(1-q \| \beta+\alpha')}.
\end{aligned}
    \end{equation}

Together, the domain condition and the two core constraints define the following feasible sets.
\begin{definition}[Feasible sets]\label{def: feasible sets}
Fix an iteration number $T$, quantile parameter $q\in(0,1)$, corruption level $\beta\in(0,1)$, and failure-probability tolerance $\delta_f\in(0,\frac12)$. Define the set of triplets satisfying the domain condition \cref{eq:alpha_constraints} and the two core constraints \cref{eq:cadv_explicit,eq:D_fail} by
\begin{equation}\label{eq:Omega_sol}
\Omega_{\mathrm{sol}}
 := \Bigl\{(D,\alpha,\alpha')\in
\mathbb{Z}_{\ge 1}\times(0,q-\beta)\times(0,1-q-\beta):
D\geq D_{\mathrm{fail}}(\alpha'),\;
\cadv(D,\alpha,\alpha')>0
\Bigr\}.
    \end{equation}
The corresponding set of feasible subsample sizes is its projection onto the $D$ coordinate:
\begin{equation}\label{eq:Omega_D}
\Omega_D
:= \Bigl\{D \in \mathbb{Z}_{\ge 1} :
\exists\,\alpha,\alpha' \text{ such that }
(D,\alpha,\alpha') \in \Omega_{\mathrm{sol}}\Bigr\}.
\end{equation}
\end{definition}
The set $\Omega_D$ contains the subsample sizes that ensure algorithmic success, namely, linear convergence with high probability. 
More precisely, for every $D\in\Omega_D$ with a feasible pair $(\alpha,\alpha')$, \cref{cor: high probability bound} yields
\[
    \|\bm{X}_T-\target\|^2
    \leq \left(1-\frac{\cadv(D,\alpha,\alpha')}{2n}\right)^T
    \|\bm{x}_0-\target\|^2
\]
with probability at least
\[
1-\delta_f-2\exp\left(-\frac{\cadv(D,\alpha,\alpha')T}{2n}\right).
\]

Intuitively, $\Omega_{\mathrm{sol}}$ and $\Omega_D$ may become empty once $\beta$ exceeds a threshold. 
To characterize this threshold, we first use the small failure probability constraint \cref{eq:D_fail} to control the required subsample size and then examine the boundary behavior of the positive contraction constraint \cref{eq:cadv_explicit}.
A convenient sufficient condition for satisfying \cref{eq:D_fail} is
\begin{equation}
    D \geq \frac{\ln(T)+\ln(1/\delta_f)-\ln(1/\beta)}{\D(1-q \| \beta+\alpha')}. \label{eq:D_constraint_failure}
\end{equation}
When $D$ is sufficiently large and remains logarithmic in $T$, \cref{eq:p_l_constraint_g_f_definition} and \cref{eq:D_constraint_failure} allow us to make $p_l$ arbitrarily small while keeping $(\alpha,\alpha')$ in the open ranges \cref{eq:alpha_constraints}.
    We therefore consider the boundary case $p_l=0$ in the contraction inequality \cref{eq:cadv_explicit}.
    In \cref{eq:cadv_explicit}, increasing $\alpha$ raises $\lrr$ and hence $g(\Phi_{\lrr})$; increasing $\alpha'$ raises $\urr$, thereby lowering $\Phi_{1-\urr}$ and the corrupted penalty $f(\Phi_{1-\urr})$ ($\Phi$, $g$, and $f$ are non-decreasing; see \cref{eq:p_l_constraint_g_f_definition}).
    Therefore, the optimistic envelope for $\cadv(D,\alpha,\alpha')$ at $p_l=0$ is approached as $(\alpha,\alpha')\to(q-\beta,1-q-\beta)$.
    Substituting these values into \cref{eq:cadv_explicit} defines the optimistic envelope for $\cadv$:
    \begin{equation} \label{eq:F_definition}
        \cadv(D,\alpha,\alpha') \leq F(q,\beta) := (1 - \beta) g(\Phi_{\frac{q - \beta}{1 - \beta}}) - \beta f(\Phi_{1 - \frac{1 - q - \beta}{1 - \beta}}).
    \end{equation}
    Given $q\in(0,1)$, $F(q,\beta)$ is defined for $\beta\in[0,\min\{q,1-q\})$. 
    The following proposition defines the maximum tolerable corruption level $\beta^*$ through the range where $F(q,\beta)$ is positive.
    \begin{prop}[Maximum tolerable corruption level]
    \label{prop: maximum corruption level}
     Given $q\in(0,1)$, $F(q,\beta)$ is non-increasing for $\beta\in[0,\min\{q,1-q\})$. Define the maximum tolerable corruption level by
    \begin{equation}\label{eq:beta_star_definition}
        \beta^*(q) := \sup\bigl\{\beta\in[0,\min\{q,1-q\}):F(q,\beta)>0\bigr\}.
    \end{equation}
    Then $\beta^*>0$ and $F(q,\beta)>0$ for every $\beta\in[0,\beta^*)$, and $\Omega_D$ defined in \cref{eq:Omega_D} is empty for every $\beta>\beta^*$.
    \end{prop}
    We first show that $F(q,\cdot)$ is non-increasing by examining how its gain and penalty terms vary with $\beta$. The supremum in \cref{eq:beta_star_definition} then determines the sign of $F(q,\beta)$ on either side of $\beta^*$. Finally, the upper bound \cref{eq:F_definition} and the domain constraints \cref{eq:alpha_constraints} imply that $\Omega_D$ is empty for every $\beta>\beta^*$.
\begin{proof}
    Fix $q$ and write
    \[
        F_q(\beta)
        =(1-\beta)g\left(\Phi_{\frac{q-\beta}{1-\beta}}\right)
        -\beta f\left(\Phi_{1-\frac{1-q-\beta}{1-\beta}}\right).
    \]
    As $\beta$ increases on its domain, $\frac{q - \beta}{1 - \beta}$ decreases and $1 - \frac{1 - q - \beta}{1 - \beta}$ increases. 
    Since $f$ and $g$ are nonnegative and non-decreasing, the gain term is non-increasing and the penalty term is non-decreasing. Hence, $F_q$ is non-increasing throughout $[0,\min\{q,1-q\})$.

    We next show that the range of $\beta$ where $F_q(\beta)>0$ is nonempty: it contains $[0, \epsilon)$ for some sufficiently small $\epsilon>0$, which implies $\beta^*>0$.
    First, $F_q(0)=g(\Phi_q)>0$ for every $q\in(0,1)$, since $Z$ has a continuous distribution with positive density on the interior of its support.  
     Moreover, $F_q(\beta)$ is continuous at $\beta=0$ since $\Phi$ is continuous on $(0,1)$, and $f$ and $g$ are continuous at $\Phi_q$. 
     Therefore, there exists some $\epsilon>0$ such that $F_q(\beta)>\frac{1}{2}g(\Phi_q)>0$ for all $\beta\in[0,\epsilon)$.
  
    For every $\beta\in[0,\beta^*)$, the definition of the supremum gives some $\hat\beta>\beta$ such that $F_q(\hat\beta)>0$. Monotonicity then gives $F_q(\beta)\geq F_q(\hat\beta)>0$.
    If $\beta\in(\beta^*,\min\{q,1-q\})$, then $F_q(\beta)\leq0$ by the definition of $\beta^*$. Hence, \cref{eq:F_definition} implies $\cadv\leq0$, so the positive-contraction constraint cannot hold and $\Omega_D$ is empty. For $\beta\geq\min\{q,1-q\}$, the domain constraints in \cref{eq:alpha_constraints} cannot hold, so $\Omega_D$ is again empty. Therefore, $\Omega_D$ is empty for every $\beta>\beta^*$.
    \end{proof}
     The bound \cref{eq:F_definition} for $\cadv$ is actually tight: for $\beta<\beta^*$, the next proposition shows that $\Omega_D$ is nonempty and contains some $D=O(\log T)$.
    Thus, a logarithmic subsample size suffices to guarantee high-probability convergence, and the minimum required subsample size $D^*$ is well defined.

\begin{prop}
\label{prop: explicit subsample size}
Given $q\in(0,1)$, let $\beta^*$ be defined as in
\cref{prop: maximum corruption level}. For every $\beta\in(0,\beta^*)$,
$\delta_f\in(0,\frac12)$, and integer $T\geq1$, let $\Omega_D$ be the feasible
set defined by \cref{eq:Omega_D}. Then there exist constants $C_1,C_2>0$ such
that every integer $D$ satisfying
\begin{equation}\label{eq:explicit_D_bound}
    D\geq
    \max\left\{C_1\bigl(\ln T+\ln(1/\delta_f)\bigr),\;C_2\right\}
\end{equation}
belongs to $\Omega_D$. Consequently, the minimum required subsample size
\begin{equation}\label{eq: D formula}
    D^*:=\min\Omega_D
\end{equation}
is well defined. Moreover, every integer $D\geq D^*$ belongs to $\Omega_D$.
\end{prop}

First, \cref{prop: maximum corruption level} gives $F(q,\beta)>0$ for $\beta\in(0,\beta^*)$. 
The proof chooses an auxiliary corruption level between $\beta$ and $\beta^*$
and uses it to construct a strictly interior pair $(\bar\alpha,\bar\alpha')$ and the corresponding $\cadv(D,\bar\alpha,\bar\alpha')$.
The positivity of the two associated KL divergences ensures that a constant lower bound on $D$ controls $p_l$ in \cref{eq:p_l_constraint_g_f_definition}, thus ensuring $\cadv(D,\bar\alpha,\bar\alpha') > 0$ in \cref{eq:cadv_explicit}; the failure probability is controlled by $D=O(\ln T+\ln(1/\delta_f))$ in \cref{eq:D_fail}. 
Thus, a sufficiently large $D$ of logarithmic order satisfies the two core constraints in \cref{def: feasible sets} and belongs to $\Omega_D$. Finally, the nonemptiness of $\Omega_D$ and the well-ordering principle yield $D^*=\min\Omega_D$, and the two core constraints remain satisfied as $D$ increases.

\begin{proof}
Choose
\[
    \bar\beta:=\frac{\beta+\beta^*}{2},\qquad
    \bar\alpha:=(1-\beta)\frac{q-\bar\beta}{1-\bar\beta},\qquad
    \bar\alpha':=(1-\beta)\frac{1-q-\bar\beta}{1-\bar\beta}.
\]
Since $\beta<\bar\beta<\beta^*$, \cref{prop: maximum corruption level} gives
$F(q,\bar\beta)>0$. 
The parameters $\bar\alpha$ and $\bar\alpha'$ are positive
and satisfy
\begin{align*}
    q-\beta-\bar\alpha
    =\frac{(1-q)(\bar\beta-\beta)}{1-\bar\beta}>0,\quad
    1-q-\beta-\bar\alpha'
    =\frac{q(\bar\beta-\beta)}{1-\bar\beta}>0.
\end{align*}
Thus, $(\bar\alpha,\bar\alpha')$ lies in the open domain
\cref{eq:alpha_constraints}. 
By construction,
\[
    \frac{\bar\alpha}{1-\beta}
    =\frac{q-\bar\beta}{1-\bar\beta},
    \qquad
    1-\frac{\bar\alpha'}{1-\beta}
    =1-\frac{1-q-\bar\beta}{1-\bar\beta}.
\]
Therefore, when $p_l=0$, the contraction expression in
\cref{eq:cadv_explicit} equals
\begin{align*}
    &(1-\beta)
    g\left(\Phi_{\frac{q-\bar\beta}{1-\bar\beta}}\right)
    -\beta
    f\left(\Phi_{1-\frac{1-q-\bar\beta}{1-\bar\beta}}\right)\\
    &\quad=F(q,\bar\beta)
    +(\bar\beta-\beta)
    \left[
        g\left(\Phi_{\frac{q-\bar\beta}{1-\bar\beta}}\right)
        +f\left(\Phi_{1-\frac{1-q-\bar\beta}{1-\bar\beta}}\right)
    \right]\\
    &\quad\geq F(q,\bar\beta)>0,
\end{align*}
where the last inequality follows from the nonnegativity of $f$ and $g$ in \cref{eq:p_l_constraint_g_f_definition}.
    At this point, both $\D(q \| \beta+\bar\alpha)$ and $\D(1-q \| \beta+\bar\alpha')$ are positive constants. Hence, by \cref{eq:p_l_constraint_g_f_definition}, there exists a constant $C_2>0$ such that, whenever $D\ge C_2$,
    \begin{equation} \label{eq:positive_contraction_in_proof}
         \cadv(D,\bar\alpha,\bar\alpha')
        \geq F(q,\bar\beta)
        -\exp\left(-\D(q \| \beta+\bar\alpha)D\right)
        (1-\beta)g\left(\Phi_{\frac{\bar\alpha}{1-\beta}}\right)> \frac{1}{2}F(q,\bar\beta) > 0,
    \end{equation} 
    and \cref{eq:cadv_explicit} is satisfied.
    Similarly, \cref{eq:D_fail} is satisfied when
    \[
        D\ge C_1\bigl(\ln T+\ln(1/\delta_f)\bigr)
    \]
    for some constant $C_1>0$. By \cref{def: feasible sets}, $(D,\bar\alpha,\bar\alpha')\in\Omega_{\mathrm{sol}}$ when $D\ge\max\{C_1(\ln T+\ln(1/\delta_f)),C_2\}$, proving that $\Omega_{\mathrm{sol}}$ and $\Omega_D$ are nonempty and establishing \cref{eq:explicit_D_bound}.

 Since $\Omega_D$ is a nonempty subset of $\mathbb{Z}_{\ge 1}$, it has a minimum by the well-ordering principle. Hence, $D^*=\min\Omega_D$ in \cref{eq: D formula} is well defined. Since $D^*\in\Omega_D$, there exists a feasible pair $(\alpha^*,\alpha'^*)$ such that $(D^*,\alpha^*,\alpha'^*)\in\Omega_{\mathrm{sol}}$.
    For any integer $D\ge D^*$, the constraint \cref{eq:D_fail} remains satisfied because $D\ge D^*\ge D_{\mathrm{fail}}(\alpha'^*)$. Moreover, $p_l(D,\alpha^*)$ decreases exponentially with $D$, so $\cadv(D,\alpha^*,\alpha'^*)$ increases with $D$ and thus \cref{eq:cadv_explicit} remains satisfied. As a result, $(D,\alpha^*,\alpha'^*)\in\Omega_{\mathrm{sol}}$, and therefore $D\in\Omega_D$.
\end{proof}

\begin{remark}
    \label{rem:D-star-parameters}
    Although the minimum subsample size $D^*$ in \cref{prop: explicit subsample size} is formally a function of $(q,\beta,T,\delta_f)$, its dependence on $\delta_f$ is mild. For any fixed sufficiently small $\delta_f$ (say, $\delta_f=0.01$), the term $\log(1/\delta_f)$ in the failure probability constraint \cref{eq:D_constraint_failure} is constant and hence is $o(\log T)$ as $T\to\infty$. Thus, for fixed $\delta_f$, the sufficient subsample size in \cref{eq:explicit_D_bound} remains of order $O(\log T)$.
    Moreover, for every $D$ satisfying \cref{eq:explicit_D_bound}, the construction in the proof of \cref{prop: explicit subsample size} gives
    $
        \cadv(D,\bar\alpha,\bar\alpha')
        > \frac{1}{2}F(q,\bar\beta)>0,
    $
    as in \cref{eq:positive_contraction_in_proof}, where $\bar\beta=(\beta+\beta^*(q))/2$. In particular, this lower bound depends only on $(q,\beta)$ and is independent of $(T,\delta_f)$.
    \end{remark}

\subsection{Oblivious Noise Model}

Compared with the Massart noise model analyzed in the previous section, in which the corruption values can be adversarially chosen to maximize damage, the oblivious model assumes that they are independent of everything else (though the realized values remain arbitrary); see \cref{def:sparse-corruption}.
Under the oblivious noise model, probabilities and expectations are taken with respect to the randomness of the relevant sampling rows, corruption indicators, and corruption values $(\ba_{\ell,j}, \cv_{\ell,j}, \epsilon_{\ell,j})_{j=0}^D$, where $\ell$ indexes the relevant iterations.
Unlike in the Massart case, the following successful events, corresponding to those in \cref{eq: successful event,eq: S_{k+1}^*}, are defined using the actual random quantile, without taking the supremum or infimum over the subsampled corruption values $(\epsilon_{k+1,j})_{j=1}^D$.
Given $\bx_k$, recall that $Q_{q,k+1}$ is the subsampled quantile in \cref{alg:streaming-qrk}. For $\tilde{m}>0$ and $\alpha_0\in[0,1]$, define the following successful events:
\begin{equation}\label{eq: oblivious successful event}
\begin{aligned}
S_{k+1}&:=\left\{Q_{q,k+1}\leq\tilde{m}\frac{\|\be_k\|}{\sqrt{n}}\right\}\cup\{\cv_{k+1,0}=0\}, \\
S_{k+1}^*&:=\left\{Q_{q,k+1}\geq\Phi_{\alpha_0}\frac{\|\be_k\|}{\sqrt{n}}\right\}.
\end{aligned}
\end{equation}
When $\be_k\neq0$, define the scaled quantile
\begin{equation}\label{eq: scaled quantile}
\tilde{Q}_{k+1}:=\frac{\sqrt{n}\,Q_{q,k+1}}{\|\be_k\|}.
\end{equation}
Then the events in \cref{eq: oblivious successful event} are equivalently $S_{k+1}=\{\tilde{Q}_{k+1}\leq\tilde{m}\}\cup\{\cv_{k+1,0}=0\}$ and $S_{k+1}^*=\{\tilde{Q}_{k+1}\geq\Phi_{\alpha_0}\}$.

This independence permits a unified analysis of corrupted and uncorrupted updates at the same threshold $Q_{q,k+1}$.
For simplicity, when $\be_k\neq0$, we use the scaled quantile $\tilde{Q}_{k+1}=\tilde{q}$ for both types of updates; the case $\be_k=0$ is handled separately in the proof of \cref{lem: oblivious one step}.
For corrupted updates, this yields the sharper bound $\fobv(\tilde{q})$, rather than the pessimistic adversarial bound $f(\tilde{m})$, for $\tilde{q}\leq\tilde{m}$; see \cref{lem: oblivious corrupted increase}. Consequently, the resulting joint analysis in \cref{lem: oblivious one step} yields tighter guarantees, generally allowing for a larger maximum tolerable corruption level $\beta^*_{\mathrm{obl}}$ and a smaller minimum subsample size $D^*_{\mathrm{obl}}$ than under Massart noise, which will be illustrated in \cref{sec: numerical results}.

\textbf{Corrupted update ($\cv_{k+1,0}=1$).}
We first introduce the error-increase function for the oblivious noise model; its derivation is deferred.
For $\tilde{q}\ge 0$ and $C\in\mathbb{R}$, define
\begin{equation}\label{eq:f_obv_pointwise_C}
\fobv(\tilde{q},C) := \mathbb{E}_Z\!\left[(C^2 - Z^2)\,\mathbf{1}_{|Z - C| \leq \tilde{q}}\,\mathbf{1}_{|Z| \leq |C|}\right],
\end{equation}
and set
\begin{equation} \label{eq: f_oblivious_pointwise}
\fobv(\tilde{q}) := \sup_{C \in [0,\infty)} \fobv(\tilde{q},C) = \sup_{C \in \mathbb{R}} \fobv(\tilde{q},C),
\end{equation}
where $Z$ is distributed as in \cref{eq:Z_def}.
\begin{lem}[Oblivious Corrupted Update] \label{lem: oblivious corrupted increase}
   Given $\bx_k\neq\target$, for every $0\leq\tilde{q}\leq\tilde{m}$, we have
\begin{equation}\label{eq: oblivious corrupted increase}
\mathbb{E}\left[\|\bX_{k+1}-\target\|^2\,\mathbf{1}_{S_{k+1}} \mid \bX_k=\bx_k,\, \cv_{k+1,0}=1,\, \tilde{Q}_{k+1}=\tilde{q}\right]
\le \|\be_k\|^2 + \frac{\|\be_k\|^2}{n}\,\fobv(\tilde{q}).
\end{equation}
Moreover, $\fobv(\tilde{q})$ is non-decreasing in $\tilde{q}$, satisfies $\lim_{\tilde{q}\to\infty}\fobv(\tilde{q})=+\infty$, and obeys
\[
0 \leq\fobv(\tilde{q})\le f(\tilde{q})=\tilde{q}^2+2\tilde{q}\,\mathbb{E}|Z|, \qquad \forall\,\tilde{q}\geq0.
\]
\end{lem}

\begin{proof}
On $S_{k+1}$, a corrupted update contributes only when $\tilde{Q}_{k+1}\le\tilde{m}$; we condition on $\tilde{Q}_{k+1}=\tilde{q}\le\tilde{m}$.
By \cref{def:sparse-corruption}, $\epsilon_{k+1,0}$ is independent of $\tilde{Q}_{k+1}$, so the same calculation as in \cref{lem: corrupted increase} gives
\begin{align}
&\mathbb{E}\left[\|\bX_{k+1}-\target\|^2\,\mathbf{1}_{S_{k+1}} \mid \bX_k=\bx_k,\, \cv_{k+1,0}=1,\, \tilde{Q}_{k+1}=\tilde{q},\, \epsilon_{k+1,0}=\epsilon\right]  \nonumber\\
&= \|\be_k\|^2 + \mathbb{E}\left[(\epsilon^2 - \langle \bm{a}_{\updateid{k+1}}, \be_k \rangle^2)\,\mathbf{1}_{|\langle \bm{a}_{\updateid{k+1}}, \be_k \rangle - \epsilon| \leq \frac{\tilde{q}\|\be_k\|}{\sqrt{n}}} \mid \bX_k=\bx_k,\, \cv_{k+1,0}=1,\,\tilde{Q}_{k+1}=\tilde{q},\, \epsilon_{k+1,0}=\epsilon
\right] \nonumber\\
& \stackrel{(a)}{=} \|\be_k\|^2 + \mathbb{E}\left[(\epsilon^2 - \langle \bm{a}_{\updateid{k+1}}, \be_k \rangle^2)\,\mathbf{1}_{|\langle \bm{a}_{\updateid{k+1}}, \be_k \rangle - \epsilon| \leq \frac{\tilde{q}\|\be_k\|}{\sqrt{n}}} 
\right] \nonumber\\ 
&\leq \|\be_k\|^2 + \frac{\|\be_k\|^2}{n}\,
\mathbb{E}\!\left[\left((\epsilon \tfrac{\sqrt{n}}{\|\be_k\|})^2 - Z^2\right)\mathbf{1}_{|Z - \epsilon \frac{\sqrt{n}}{\|\be_k\|}|\leq \tilde{q}}\,\mathbf{1}_{|Z| \le |\epsilon \frac{\sqrt{n}}{\|\be_k\|}|}\right],
\label{eq: error increase final}
\end{align}
where (a) uses the independence of $\ba_{k+1,0}$ from the other variables, while the last inequality uses the definition of $Z$ in \cref{eq:Z_def} and discards the nonpositive contribution on
$\left\{|Z| > \left|\epsilon \frac{\sqrt{n}}{\|\be_k\|}\right|\right\}$.
With $C=\epsilon\sqrt{n}/\|\be_k\|$, the expectation in \cref{eq: error increase final} equals $\fobv(\tilde{q},C)$, so taking the supremum over $C\in [0, +\infty)$ (using evenness of $\fobv(\tilde{q},C)$ in $C$) yields \eqref{eq: oblivious corrupted increase}.

For the properties of $\fobv$, note that $(C^2-Z^2)\mathbf{1}_{|Z|\le |C|}\ge 0$, so enlarging the acceptance set $\{|Z-C|\le\tilde{q}\}$ increases $\fobv(\tilde{q},C)$; hence $\fobv(\tilde{q},C)$ and its supremum $\fobv(\tilde{q})$ are non-decreasing in $\tilde{q}$.
For $\tilde{q}\ge 0$, on $\{|Z|\le |C|,\,|Z-C|\le\tilde{q}\}$ with $C\ge 0$ we have 
\begin{align*}
    0\le C-|Z|&\le\tilde{q}\\
    C+|Z|&\le 2|Z|+\tilde{q},
\end{align*}
hence
$C^2-Z^2\le \tilde{q}(2|Z|+\tilde{q})$.
Taking expectations gives $\fobv(\tilde{q},C)\le \tilde{q}^2+2\tilde{q}\,\mathbb{E}|Z|=f(\tilde{q})$, and therefore $\fobv(\tilde{q})\le f(\tilde{q})$.
Finally, for $C=\tilde{q}\ge 0$, the indicators $\mathbf{1}_{|Z|\le C}$ and $\mathbf{1}_{|Z-C|\le\tilde{q}}$ both equal~$1$ exactly when $0\le Z\le\tilde{q}$, so
\begin{align*}
    \fobv(\tilde{q})&\ge \fobv(\tilde{q},\tilde{q})\\
    &\ge \mathbb{E}\!\left[(\tilde{q}^2-Z^2)\mathbf{1}_{0\le Z\le\tilde{q}}\right] \\ 
&=\tilde{q}^2\,\mathbb{P}(0\le Z\le\tilde{q})-\mathbb{E}[Z^2\mathbf{1}_{0\le Z\le\tilde{q}}].
\end{align*}
Since $Z$ is distributed as in \cref{eq:Z_def}, $\mathbb{P}(0\le Z\le\tilde{q})\to \mathbb{P}(Z\ge 0)>0$ while the second term stays bounded, hence $\fobv(\tilde{q})\to+\infty$ as $\tilde{q}\to\infty$. %
\end{proof}
By contrast, in \cref{lem: corrupted increase}, the adversary may set
\begin{equation} \label{eq: adversarial worst-case noise}
\epsilon_{k+1,0} = \langle\ba_{\updateid{k+1}}, \be_k\rangle + Q_{q,k+1}\operatorname{sign}(\langle\ba_{\updateid{k+1}}, \be_k\rangle),
\end{equation}
thereby forcing the update to be accepted with the largest error increase and yielding the worst-case bound $f(\tilde{m})$.
Since $\fobv(\tilde{q})\le f(\tilde{q}) \leq f(\tilde{m})$ for $\tilde{q} \leq \tilde{m}$, the oblivious bound is never worse than the adversarial one.

\textbf{Uncorrupted update ($\cv_{k+1,0}=0$).}
Fix $\tilde{Q}_{k+1}=\tilde{q}$. Since $\cv_{k+1,0}=0$, we have $\mathbf{1}_{S_{k+1}}=1$, and the update is accepted if and only if $|Z|\leq \tilde{q}$.
On acceptance,
$
\|\be_{k+1}\|^2 = \|\be_k\|^2 - \langle \ba_{\updateid{k+1}}, \be_k\rangle^2;
$
if the update is rejected, we have $\|\be_{k+1}\|^2 = \|\be_k\|^2$.
Therefore,
\begin{equation}\label{eq: error decrease final}
\mathbb{E}\left[\left\|\boldsymbol{X}_{k+1}-\target\right\|^2  \mid \boldsymbol{X}_k=\boldsymbol{x}_k, \cv_{k+1,0} = 0, \tilde{Q}_{k+1}=\tilde{q}\right]
= \|\be_k\|^2 - \frac{\|\be_k\|^2}{n}\,\mathbb{E}\left[Z^2 \mathbf{1}_{|Z| \leq \tilde{q}}\right].
\end{equation}
We define the error decrease at threshold $\tilde{q}$ by
\begin{equation}\label{eq: g_oblivious_pointwise}
g(\tilde{q}) := \mathbb{E}\left[Z^2 \mathbf{1}_{|Z| \leq \tilde{q}}\right],
\end{equation}
which matches the Massart definition $g(\Phi_{\alpha_0})=\truncZsq$ in \cref{lem: error decreased}. 

Combining the two cases, we obtain a joint one-step bound that accounts for both corrupted and uncorrupted updates at the same threshold $Q_{q,k+1}$.
\begin{lem}[Joint one-step bound, oblivious noise] \label{lem: oblivious one step}
    Given $\bx_k$, let $\tilde{m}>0$ and $\alpha_0,p_l\in[0,1]$, and let $S_{k+1}$ and $S_{k+1}^*$ be the events in \cref{eq: oblivious successful event}. Suppose that $\mathbb{P}((S_{k+1}^*)^c)\leq p_l$ and $\Phi_{\alpha_0}\leq\tilde{m}$. Then,
    \begin{align*}
        &\mathbb{E}\!\left[\|\bX_{k+1} - \target\|^2\,\mathbf{1}_{S_{k+1}} \mid \bX_k = \bx_k\right] \\
        &\leq \|\be_k\|^2 - \frac{\|\be_k\|^2}{n}\left(\inf_{\tilde{q}\,\in\,[\Phi_{\alpha_0},\,\tilde{m}]}
        \!\bigl((1-\beta)\,g(\tilde{q}) - \beta\,\fobv(\tilde{q})\bigr)
       - (1-\beta)p_l\,g(\Phi_{\alpha_0})\right).
    \end{align*}
  
\end{lem}

\begin{proof}
    If $\be_k=0$, then on $S_{k+1}$, either the update is uncorrupted and has zero residual, or $Q_{q,k+1}=0$, in which case any accepted corrupted update also has zero residual. Thus, $\bX_{k+1}=\target$ on $S_{k+1}$, and the claim follows. Hence, assume $\be_k\neq0$ below and let $\tilde{Q}_{k+1}$ be the scaled quantile in \cref{eq: scaled quantile}.

    Decomposing the expectation by $\cv_{k+1,0}$ and using the tower property
    $\mathbb{E}=\mathbb{E}_{\tilde{Q}_{k+1}}
    \mathbb{E}[\,\cdot\mid \tilde{Q}_{k+1}]$, we apply
    \cref{eq: error decrease final,eq: g_oblivious_pointwise} for the uncorrupted case
    ($\cv_{k+1,0}=0$) and \cref{lem: oblivious corrupted increase} for the corrupted case
    ($\cv_{k+1,0}=1$):
    \begin{align}
        &\mathbb{E}\!\left[
        \|\bX_{k+1}-\target\|^2\mathbf{1}_{S_{k+1}}
        \mid \bX_k=\bx_k
        \right]
        \nonumber\\
        &\leq
        \beta\,
        \mathbb{E}\!\left[
        \mathbf{1}_{\tilde{Q}_{k+1}\leq\tilde{m}}
        \left(
        \|\be_k\|^2
        +
        \frac{\|\be_k\|^2}{n}\fobv(\tilde{Q}_{k+1})
        \right)
        \right]
        \nonumber\\
        &\qquad
        +(1-\beta)\,
        \mathbb{E}\!\left[
        \|\be_k\|^2
        -
        \frac{\|\be_k\|^2}{n}g(\tilde{Q}_{k+1})
        \right]
        \nonumber\\
        &\leq
        \|\be_k\|^2
        -
        \frac{\|\be_k\|^2}{n}
        \mathbb{E}\!\left[
        (1-\beta)g(\tilde{Q}_{k+1})
        -
        \beta\mathbf{1}_{\tilde{Q}_{k+1}\leq\tilde{m}}
        \fobv(\tilde{Q}_{k+1})
        \right].  \label{eq: error split}
    \end{align}
    Let
    \[
        c_0
        :=
        \inf_{\tilde{q}\in[\Phi_{\alpha_0},\tilde{m}]}
        \bigl((1-\beta)g(\tilde{q})-\beta\fobv(\tilde{q})\bigr).
    \]
    We lower-bound the integrand in \cref{eq: error split} pointwise.
    If $\tilde{Q}_{k+1}\in[\Phi_{\alpha_0},\tilde{m}]$, it is at least $c_0$ by definition.
    If $\tilde{Q}_{k+1}>\tilde{m}$, then, using the monotonicity and the non-negativity of $g$ and $\fobv$,
    \[
        (1-\beta)g(\tilde{Q}_{k+1})
        \geq
        (1-\beta)g(\tilde{m})-\beta\fobv(\tilde{m})
        \geq c_0.
    \]
    Finally, if $\tilde{Q}_{k+1}<\Phi_{\alpha_0}$, then $\Phi_{\alpha_0}\leq\tilde{m}$ and the monotonicity of $\fobv$ gives
    \begin{align*}
        &(1-\beta)g(\tilde{Q}_{k+1})
        -\beta\fobv(\tilde{Q}_{k+1})\\
        &\geq
        -\beta\fobv(\Phi_{\alpha_0})\\
        &=
        (1-\beta)g(\Phi_{\alpha_0})
        -\beta\fobv(\Phi_{\alpha_0})
        -(1-\beta)g(\Phi_{\alpha_0})\\
        &\geq
        c_0-(1-\beta)g(\Phi_{\alpha_0}).
    \end{align*}
    Therefore,
    \[
        (1-\beta)g(\tilde{Q}_{k+1})
        -
        \beta\mathbf{1}_{\tilde{Q}_{k+1}\leq\tilde{m}}
        \fobv(\tilde{Q}_{k+1})
        \geq
        c_0
        -
        (1-\beta)g(\Phi_{\alpha_0})
        \mathbf{1}_{\{\tilde{Q}_{k+1}<\Phi_{\alpha_0}\}}.
    \]
    Taking expectations and using
    $\mathbb{P}(\tilde{Q}_{k+1}<\Phi_{\alpha_0})\leq p_l$ gives
    \[
        \mathbb{E}\!\left[
        (1-\beta)g(\tilde{Q}_{k+1})
        -
        \beta\mathbf{1}_{\tilde{Q}_{k+1}\leq\tilde{m}}
        \fobv(\tilde{Q}_{k+1})
        \right]
        \geq
        c_0-(1-\beta)p_lg(\Phi_{\alpha_0}).
    \]
    Substituting the definition of $c_0$ into \cref{eq: error split} completes the proof.
\end{proof}

The oblivious noise model yields results parallel to \cref{thm: convergence in expectation-streaming}, \cref{prop: maximum corruption level}, and \cref{prop: explicit subsample size} for convergence, maximum tolerable corruption level, and subsample size.
\begin{thm}[Convergence, oblivious noise] \label{thm: convergence-oblivious}
    Let the setting be as in \cref{thm: convergence in expectation-streaming}, except that the corruption values $(\epsilon_{k+1,j})_{j=0}^D$ are oblivious.
    Let $\tau_2$ be the stopping time
   $
 \tau_2 =\inf \left\{t \geq 1: S_t^c\right\}
$,
where $S_t$ is the successful event defined in \cref{eq: oblivious successful event} evaluated at the iterate
$\bm{X}_{t-1}$, with $\tilde{m} = \Phi_{1 - \urr}$.
    Choose parameters $\alpha \in (0, q-\beta)$ and $\alpha' \in (0, 1-q-\beta)$, and set
    $p_l = \exp(-\D(q \| \beta + \alpha)D)$, such that
    \begin{align} \label{eq: oblivious contraction}
        \cobl(D,\alpha,\alpha') := \inf_{\tilde{q}\,\in\,[\Phi_{\lrr},\,\Phi_{1 - \urr}]}
        \!\bigl((1-\beta)\,g(\tilde{q}) - \beta\,\fobv(\tilde{q})\bigr)
        - p_l\,(1-\beta)\,g(\Phi_{\lrr}) > 0.
    \end{align}
    Here
    $$
    \fobv(t) = \sup_{C \in [0,\infty)} \mathbb{E}_Z \left[ (C^2 - Z^2) \mathbf{1}_{|Z| \leq C} \mathbf{1}_{|Z - C| \leq t} \right], \quad
    g(t) = \mathbb{E}_Z \left[ Z^2 \mathbf{1}_{|Z| \leq t} \right], \quad \forall t \in [0,\infty),
    $$
    where $Z$ is defined as in \cref{eq:Z_def}.
    Then
    $$
    \mathbb{E}\left(\left\|\bm{X}_T - \target\right\|^2 \mathbf{1}_{\tau_2 > T}\right)
    \leq
    \left(1-\frac{\cobl}{n}\right)^T
    \left\|\init - \target\right\|^2,
    $$
    and the failure probability $\mathbb{P}(\tau_2 \leq T)$ satisfies the bound in \cref{eq:failure-prob}.
\end{thm}
\begin{proof}
The proof follows the same structure as that of
\cref{thm: convergence in expectation-streaming}. Applying
\cref{lem: oblivious one step} with
$\alpha_0=\lrr$ and $\tilde{m}=\Phi_{1-\urr}$ gives the one-step
contraction \cref{eq: oblivious contraction}.
Since \cref{lem: sub quantile streaming - simple} also holds under the weaker
oblivious noise model, the same argument as in the proof of
\cref{thm: convergence in expectation-streaming} gives the
failure-probability bound \cref{eq:failure-prob}. Iterating the one-step
contraction then yields the stated convergence bound.
\end{proof}
The key difference compared to \cref{thm: convergence in expectation-streaming} is that the positive-contraction condition 
$$
\cadv = (1-\beta)\,g(\Phi_{\lrr}) - \beta\,f(\Phi_{1-\urr}) - p_l(1-\beta)\,g(\Phi_{\lrr}) > 0
$$
in \cref{eq: positive contraction} becomes the oblivious form $\cobl(D,\alpha,\alpha')$ given in \cref{eq: oblivious contraction}.
In parallel with \cref{def: feasible sets}, define the oblivious feasible sets by replacing $\cadv(D,\alpha,\alpha')>0$ with $\cobl(D,\alpha,\alpha')>0$ in \cref{eq: oblivious contraction}, while retaining the domain condition \cref{eq:alpha_constraints} and the small failure-probability constraint \cref{eq:D_fail}:
\begin{align}
\Omega_{\mathrm{sol,obl}}
&:= \Bigl\{(D,\alpha,\alpha') \in
\mathbb{Z}_{\ge 1}\times(0,q-\beta)\times(0,1-q-\beta):
D\geq D_{\mathrm{fail}}(\alpha'),\;
\cobl(D,\alpha,\alpha')>0\Bigr\}, \label{eq:Omega_sol_oblivious}\\
\Omega_{D,\mathrm{obl}}
&:=\Bigl\{D\in\mathbb{Z}_{\ge 1}:
\exists\,\alpha,\alpha'\text{ such that }
(D,\alpha,\alpha')\in\Omega_{\mathrm{sol,obl}}\Bigr\}. \label{eq:Omega_D_oblivious}
\end{align}
For every $(\alpha,\alpha')$ in the domain, the infimum interval in \cref{eq: oblivious contraction} contains
\begin{equation} \label{eq: infimum interval Ibeta}
I_\beta := \left[\Phi_{\frac{q-\beta}{1-\beta}},\,\Phi_{1-\frac{1-q-\beta}{1-\beta}}\right], \quad \beta \in [0,\min\{q,1-q\}).
\end{equation}
Hence $\cobl(D,\alpha,\alpha')$ is bounded above by evaluating the infimum on this boundary interval and setting $p_l=0$. 
In parallel with \cref{eq:F_definition}, we obtain an upper bound on $\cobl$:
\begin{align}
\cobl(D,\alpha,\alpha') &\leq \inf_{\tilde{q}\,\in\,[\Phi_{\frac{q-\beta}{1-\beta}},\,\Phi_{1-\frac{1-q-\beta}{1-\beta}}]}
\!\bigl((1-\beta)\,g(\tilde{q}) - \beta\,\fobv(\tilde{q})\bigr) - p_l\,(1-\beta)\,g(\Phi_{\lrr}) \nonumber\\
&\leq \inf_{\tilde{q}\,\in\,[\Phi_{\frac{q-\beta}{1-\beta}},\,\Phi_{1-\frac{1-q-\beta}{1-\beta}}]}
\!\bigl((1-\beta)\,g(\tilde{q}) - \beta\,\fobv(\tilde{q})\bigr) =: F_{\mathrm{obl}}(q,\beta). \label{eq:F_obl_definition}
\end{align}
We next compare the oblivious and Massart contraction rates. Specifically, for all $(D,\alpha,\alpha')$ in the domain,
\begin{equation}\label{eq: comparison}
\cobl(D,\alpha,\alpha') \geq \cadv(D,\alpha,\alpha'), 
\quad F_{\mathrm{obl}}(q,\beta) \geq F(q,\beta).
\end{equation}
Both inequalities follow from the monotonicity of $g$ and $\fobv$ and the bound $\fobv(t)\leq f(t)$: for $0\leq l\leq r$ and $t\in[l,r]$,
\[
g(t) \geq g(l)
\quad\text{and}\quad
\fobv(t) \leq \fobv(r) \leq f(r).
\]
In the first inequality, take $l=\Phi_{\lrr}$ and $r=\Phi_{1-\urr}$; in the second, take $l$ and $r$ to be the endpoints of $I_\beta$ in \cref{eq: infimum interval Ibeta}.
With the optimistic envelope $F_{\mathrm{obl}}$ in \cref{eq:F_obl_definition} and the comparison \cref{eq: comparison} in place, we characterize the maximum tolerable corruption level under oblivious noise in parallel with \cref{prop: maximum corruption level}.
\begin{prop}
[Maximum tolerable corruption level, oblivious noise] \label{prop: maximum corruption level-oblivious}
Given $q\in(0,1)$, $F_{\mathrm{obl}}(q,\beta)$ is non-increasing for $\beta\in[0,\min\{q,1-q\})$. Define the maximum tolerable corruption level by
\begin{equation}\label{eq:beta_star_definition_oblivious}
    \beta^*_{\mathrm{obl}}(q)
    :=\sup\bigl\{\beta\in[0,\min\{q,1-q\}):F_{\mathrm{obl}}(q,\beta)>0\bigr\}.
\end{equation}
Then $\beta^*_{\mathrm{obl}}(q)>0$ and $F_{\mathrm{obl}}(q,\beta)>0$ for every $\beta\in[0,\beta^*_{\mathrm{obl}}(q))$, and $\Omega_{D,\mathrm{obl}}$ defined in \cref{eq:Omega_D_oblivious} is empty for every $\beta>\beta^*_{\mathrm{obl}}(q)$.
\end{prop}
The main new point compared with \cref{prop: maximum corruption level}
is the monotonicity argument: $F_{\mathrm{obl}}$ is an infimum over an
interval that expands as $\beta$ increases.
\begin{proof}
Let
$I_\beta$ be defined as in \cref{eq: infimum interval Ibeta}.
For $\beta_1\leq\beta_2$, we have $I_{\beta_1}\subseteq I_{\beta_2}$ and
\[
    (1-\beta_2)g(t)-\beta_2\fobv(t)
    \leq(1-\beta_1)g(t)-\beta_1\fobv(t)
    \qquad (t\in I_{\beta_1}),
\]
by the nonnegativity of $g$ and $\fobv$. 
Therefore,
\[
    F_{\mathrm{obl}}(q,\beta_2)
    \leq\inf_{t\in I_{\beta_1}}
    \bigl((1-\beta_2)g(t)-\beta_2\fobv(t)\bigr)
    \leq\inf_{t\in I_{\beta_1}}
    \bigl((1-\beta_1)g(t)-\beta_1\fobv(t)\bigr)
    =F_{\mathrm{obl}}(q,\beta_1),
\]
so $F_{\mathrm{obl}}(q,\cdot)$ is non-increasing.

Since $\beta^*(q)>0$ and $F(q,\beta)>0$ for every
$\beta\in[0,\beta^*(q))$ by \cref{prop: maximum corruption level}, and
$F_{\mathrm{obl}}(q,\beta)\geq F(q,\beta)$ by \cref{eq: comparison}, we have
\[
    [0,\beta^*(q))
    \subseteq
    \{\beta:F_{\mathrm{obl}}(q,\beta)>0\}.
\]
Therefore,
\[
    \beta^*_{\mathrm{obl}}(q)\geq \beta^*(q)>0.
\]
The same supremum argument as in
the proof of \cref{prop: maximum corruption level} gives 
$F_{\mathrm{obl}}(q,\beta)>0$ for every $\beta \in [0,\beta^*_{\mathrm{obl}}(q))$, and
$\Omega_{D,\mathrm{obl}}=\varnothing$ for every
$\beta>\beta^*_{\mathrm{obl}}(q)$.
\end{proof}

The next result is parallel to \cref{prop: explicit subsample size}. The only new step is to compare the oblivious contraction over the entire quantile interval; the failure-probability and minimum-size arguments are unchanged.
\begin{prop}[Explicit subsample size, oblivious noise]
\label{prop: explicit subsample size-oblivious}
Given $q\in(0,1)$, let $\beta^*_{\mathrm{obl}}(q)$ be defined as in
\cref{prop: maximum corruption level-oblivious}. For every
$\beta\in(0,\beta^*_{\mathrm{obl}}(q))$, $\delta_f\in(0,\frac12)$, and
integer $T\geq1$, there exist constants $C_1,C_2>0$ such that every integer
$D$ satisfying
\begin{equation}\label{eq:explicit_D_bound_oblivious}
    D\geq
    \max\left\{C_1\bigl(\ln T+\ln(1/\delta_f)\bigr),\;C_2\right\}
\end{equation}
belongs to $\Omega_{D,\mathrm{obl}}$. Consequently, the minimum required
subsample size
\begin{equation}\label{eq: D formula oblivious}
    D^*_{\mathrm{obl}}:=\min\Omega_{D,\mathrm{obl}}
\end{equation}
is well defined. Moreover, every integer $D\geq D^*_{\mathrm{obl}}$ belongs
to $\Omega_{D,\mathrm{obl}}$.
\end{prop}

\begin{proof}
Choose $\bar\beta=(\beta+\beta^*_{\mathrm{obl}}(q))/2$, and define
$(\bar\alpha,\bar\alpha')$ as in the proof of
\cref{prop: explicit subsample size}. The same calculation shows that this
pair lies in the open domain \cref{eq:alpha_constraints}, its two associated
KL divergences are positive, and its quantile interval is $I_{\bar\beta}$ defined as in \cref{eq: infimum interval Ibeta}.
The only difference from the Massart case is the following pointwise comparison:
for every $t\in I_{\bar\beta}$,
\[
    (1-\beta)g(t)-\beta\fobv(t)
    =(1-\bar\beta)g(t)-\bar\beta\fobv(t)
    +(\bar\beta-\beta)\bigl(g(t)+\fobv(t)\bigr).
\]
Since $\bar\beta<\beta^*_{\mathrm{obl}}(q)$, taking the infimum over
$I_{\bar\beta}$ gives
\[
    \inf_{t\in I_{\bar\beta}}
    \bigl((1-\beta)g(t)-\beta\fobv(t)\bigr)
    \geq F_{\mathrm{obl}}(q,\bar\beta)>0.
\]

The remainder follows exactly as in the proof of
\cref{prop: explicit subsample size}: exponential decay of $p_l$ and the
failure-probability bound yield \cref{eq:explicit_D_bound_oblivious}, and the
well-ordering and monotonicity arguments yield the properties of
$D^*_{\mathrm{obl}}$.

\end{proof}

\section{Numerical Results} \label{sec: numerical results}

This section has two parts. First, we describe how to compute the maximum tolerable corruption level $\beta^*$ and the required subsample size $D^*$. We then compare these theoretical bounds with numerical simulations. Code for the numerical computations and simulations is available at \url{https://github.com/wtree101/Explicit-beta-and-D-for-QRK}.

\subsection{Computation of \texorpdfstring{$\beta^*$ and $D^*$}{beta* and D*}} \label{sec: explicit D}

We first introduce a practical method to compute the maximum tolerable corruption level $\beta^* = \beta^*(q)$ and the required subsample size $D^* = D^*(q, \beta, T,\delta_f)$.
For simplicity, we describe the procedure under the Massart noise model, but the same approach applies to the oblivious noise model with minor modifications.

To make these constants explicit and computable, we utilize normal approximations for the quantiles and moments of $Z$, 
where $Z=\langle\sqrt{n}\bm{a},\bm{e}^{(1)}\rangle$ is defined in \cref{eq:Z_def}.
 For large $n$, $Z$ is approximately distributed as $\mathcal{N}(0,1)$. 
 Accordingly, we use the following standard normal approximations:
\begin{itemize}
    \item The $a$-quantile of $|Z|$ is approximated by $\Phi_a$, the $a$-quantile of the half-normal distribution $|\mathcal{N}(0,1)|$, i.e., $\mathbb{P}(|\mathcal{N}(0,1)|\leq\Phi_{a}) = a$. 
    \item The expected value $\mathbb{E}[|Z|] \approx \sqrt{2/\pi} \approx 0.798$.
    \item The truncated second moment $g(\Phi_{\alpha_0}) =  \mathbb{E}\left[Z^2 \mathbf{1}_{|Z| \leq \Phi_{\alpha_0}}  \right] \approx \frac{1}{\sqrt{2 \pi}} \int_{-\Phi_{\alpha_0}}^{\Phi_{\alpha_0}} x^2 e^{-x^2 / 2} dx$. 
\end{itemize}

These approximations enable explicit, computable expressions for the error increase function $f(\cdot)$, the error decrease function $g(\cdot)$, and the contraction factor $\cadv$ in \cref{eq: positive contraction}. 
Using these expressions, we compute $\beta^*$ and $D^*$ as follows:

        \begin{itemize}
            \item \textbf{Finding  $\beta^*$:} For a given $q$, since $F(q,\beta)$ is non-increasing in $\beta$, we use a binary search on $\beta \in (0, \min\{q , 1 -q\})$ to locate the supremum in \cref{eq:beta_star_definition}.
            
            \item \textbf{Finding $D^*$:} For fixed $T$, $\delta_f$, and $q$, with $\beta < \beta^*(q)$, \cref{prop: explicit subsample size} guarantees that $\Omega_D$ is nonempty and that every integer $D \ge D^*$ belongs to $\Omega_D$. Thus, choosing a sufficiently large feasible upper bound $D_{\max}$ (e.g., $D_{\max}=1000$), we can use binary search to find  $D^*$ in $[1,D_{\max}]$.
            For each candidate $D$, we check whether there exist feasible $\alpha$ and $\alpha'$ in \cref{eq:alpha_constraints} satisfying both the failure probability constraint \cref{eq:D_fail} and the contraction constraint \cref{eq:cadv_explicit}. 
            This involves first determining the largest $\alpha'$ allowed by the failure probability constraint, then searching for an $\alpha$ such that the contraction rate is positive. Since $\cadv$ in \cref{eq:cadv_explicit} is not necessarily monotonic in $\alpha$, a grid search over $\alpha \in (0, q-\beta)$ can be used. 
            
        \end{itemize}

We next illustrate this procedure with representative parameter choices. To compare two noise models, we use $q=0.75$, $\delta_f=0.1$, and $T=20000$ in both cases, and consider $\beta\in(0,0.02)$ when computing $D^*$.

\paragraph{Massart noise.}
\Cref{fig:q_beta_max_massart} shows the maximum tolerable corruption level $\beta^*(q)$, whose peak reaches approximately $0.069$ at $q=0.85$. Under the parameter choices above, \cref{fig:D_beta_curve_massart} shows the required subsample size $D^*$ as a function of $\beta$. At the representative corruption level $\beta=0.01$, we obtain $D^*=25$.

\begin{figure}[H]
    \centering
    \subfloat[Maximum tolerable corruption $\beta^*(q)$.\label{fig:q_beta_max_massart}]{%
        \includegraphics[width=.45\textwidth]{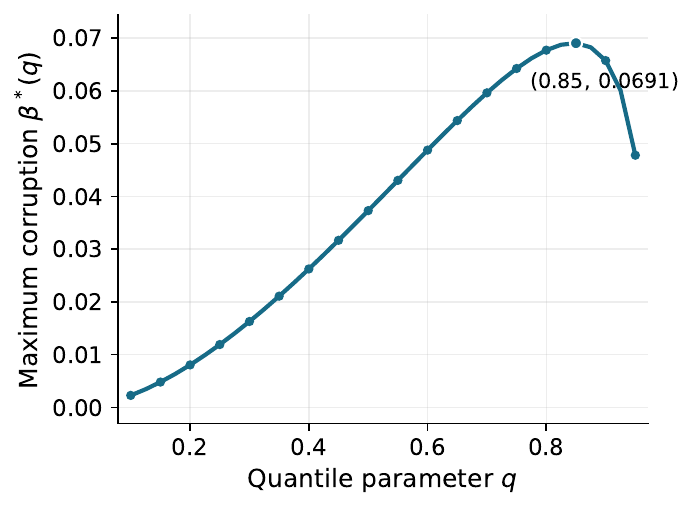}}
    \hfill
    \subfloat[Minimum subsample size $D^*(\beta)$ at $q=0.75$.\label{fig:D_beta_curve_massart}]{%
        \includegraphics[width=.45\textwidth]{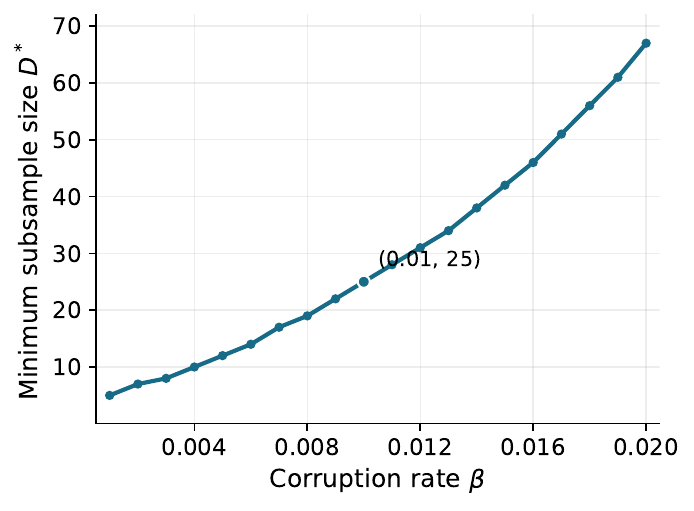}}
    \caption{Numerical bounds under Massart noise.}
    \label{fig:numerical_bounds_massart}
\end{figure}

\paragraph{Oblivious noise.}

\Cref{fig:q_beta_max_oblivious} shows the corresponding maximum tolerable corruption level $\beta^*_{\mathrm{obl}}(q)$, where the peak reaches approximately $0.320$ at $q=0.65$. Using the same parameters as in the Massart case, \cref{fig:D_beta_curve_oblivious} gives $D^*_{\mathrm{obl}}=13$ when $\beta=0.01$, compared with $D^*=25$ under Massart noise. 

\begin{figure}[H]
    \centering
    \subfloat[Maximum tolerable corruption $\beta^*_{\mathrm{obl}}(q)$.\label{fig:q_beta_max_oblivious}]{%
        \includegraphics[width=.45\textwidth]{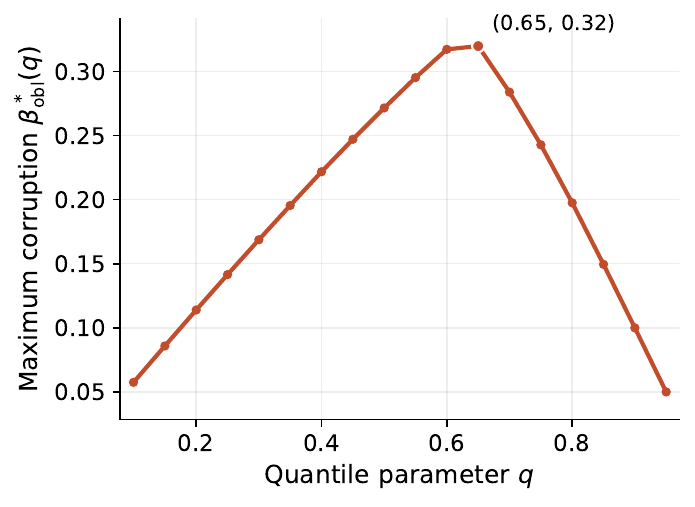}}
    \hfill
    \subfloat[Minimum subsample size $D^*_{\mathrm{obl}}(\beta)$ at $q=0.75$.\label{fig:D_beta_curve_oblivious}]{%
        \includegraphics[width=.45\textwidth]{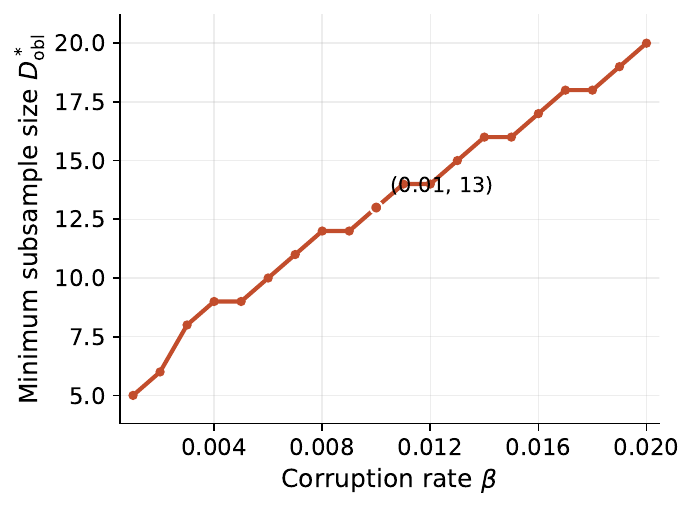}}
    \caption{Numerical bounds under oblivious noise.}
    \label{fig:numerical_bounds_oblivious}
\end{figure}

\begin{remark} \label{rem: Gaussian corruption}
    The thresholds $\beta^*$ and $D^*$ can be further improved when the distribution of the corruption is known. For Gaussian corruption, the error increase function $f_{\mathrm{obl}}(\tilde{q})$ defined through \cref{eq:f_obv_pointwise_C,eq: f_oblivious_pointwise} can be reduced by replacing the supremum over $C$ with an expectation over $C \sim \mathcal{N}(0, \sigma^2)$:
    \begin{equation*}
        f_{\mathrm{Gaussian}}(\tilde{q}) = \sup_{\sigma > 0} \int_{-\infty}^{\infty} \mathbb{E}_Z\!\left[(C^2 - Z^2)\,\mathbf{1}_{|Z - C| \leq \tilde{q}}\,\mathbf{1}_{|Z| \leq |C|}\right] \frac{1}{\sqrt{2\pi\sigma^2}}e^{-\frac{C^2}{2 \sigma^2}} \,\mathrm{d}C \leq f_{\mathrm{obl}}(\tilde{q}),
    \end{equation*}
    which yields a larger maximum tolerable corruption level $\beta^*_{\mathrm{Gaussian}}(0.60) = 0.397 > 0.320$.
\end{remark}
            
\subsection{Numerical Simulations} \label{sec: experiments}
We next examine how the subsample size $D$ affects performance and how well the model-specific theoretical bound predicts empirical success as $\beta$ and $T$ vary.

Throughout, we set $n=100$, $q=0.8$, and $\delta_f=0.1$. We use $c_{\mathrm{succ}}=0.05$ in the empirical success criterion below. This choice is representative rather than intrinsic, but it imposes a meaningful accuracy requirement over the horizons considered: the success threshold is approximately $6.7\times 10^{-3}$ at $T=10000$ and $4.5\times 10^{-5}$ at $T=20000$. A larger value of $c_{\mathrm{succ}}$ would impose an even stricter criterion.

We also specify the corruption values used in the simulations. Under the Massart noise model, we set $\epsilon_{k+1,j}=10^{15}$ for $j=1,\dots,D$ to represent an arbitrarily large corruption on the quantile subsample. For the update sample, we set the corruption value $\epsilon_{k+1,0}$ adversarially according to \cref{eq: adversarial worst-case noise}. Under the oblivious noise model, we sample $\epsilon_{k+1,j}\sim\operatorname{Unif}(-1000,1000)$ independently for $j=0,\dots,D$.

We then study the dependence on $T$ and $\beta$ under the two noise models. In the following discussion and in \Cref{fig:D_vs_T,fig:D_vs_beta}, $D^*$ denotes the corresponding theoretical bound under each model: $D^*$ under Massart noise and $D^*_{\mathrm{obl}}$ under oblivious noise. To study the relationship between $D$ and $T$, \Cref{fig:D_vs_T} reports the empirical success probability for each pair $(D,T)$. We draw and normalize a ground truth solution $x^\ast\sim N(0,I_n)$ and, for each pair, run $100$ independent trials of \Cref{alg:streaming-qrk} from $x_0=0$. A trial is successful if
$\frac{\|x_T-x^\ast\|^2}{\|x^\ast\|^2}\leq \left(1-\frac{c_{\mathrm{succ}}}{n}\right)^T$.
Both panels use $\beta=0.01$, with the theoretical bound $D^*$ overlaid as a function of $T$.
For each fixed $T$, the success rate generally approaches $1$ as $D$ increases.
Under Massart noise, $D^*$ tracks the empirical success transition well. Under oblivious noise, $D^*$ is consistent with the sporadic pairs $(D,T)$ with $D>5$ that have lower success rates, though the clearest empirical transition occurs at $D=5$.

 \begin{figure}[H]
    \centering
        \subfloat[Massart \label{fig:D_vs_T_massart}]{\includegraphics[width=.45\textwidth]{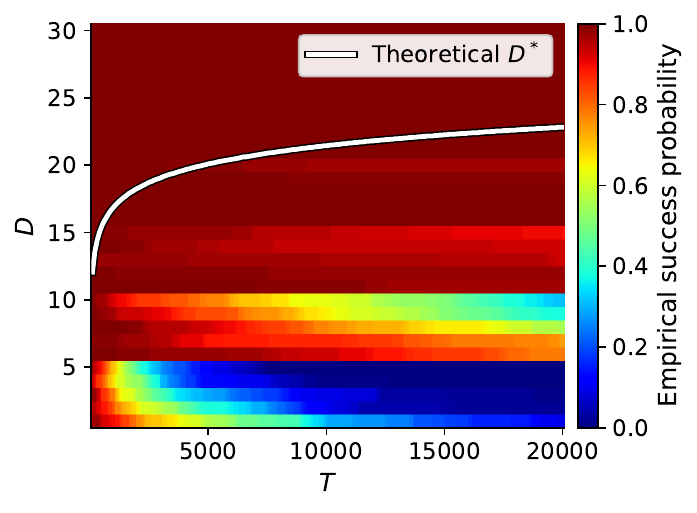}
        }\hfill
        \subfloat[Oblivious \label{fig:D_vs_T_oblivious}]{\includegraphics[width=.45\textwidth]{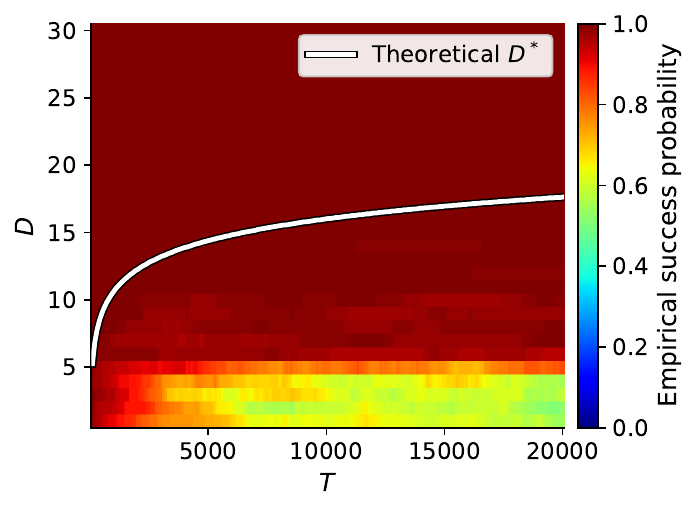}
        }\hfill
        \caption{Empirical success probability as a function of the quantile subsample size $D$ and iteration count $T$, using $c_{\mathrm{succ}}=0.05$. The solid curve shows the theoretical bound $D^*$.}
        \label{fig:D_vs_T}
\end{figure}

To study the relationship between $D$ and $\beta$, \Cref{fig:D_vs_beta} reports the empirical success probability for each pair $(D,\beta)$, again using $100$ independent trials from $x_0=0$. 
Both panels use $T=20000$, with $D^*$ overlaid as a function of $\beta$.
For each $\beta$, the success rate generally increases with $D$. 
Under both models, $D^*$ tracks the empirical success transition well, though it becomes more conservative for larger $\beta$.

\begin{figure}[H]
    \centering
        \subfloat[Massart \label{fig:D_vs_beta_massart}]{\includegraphics[width=.45\textwidth]{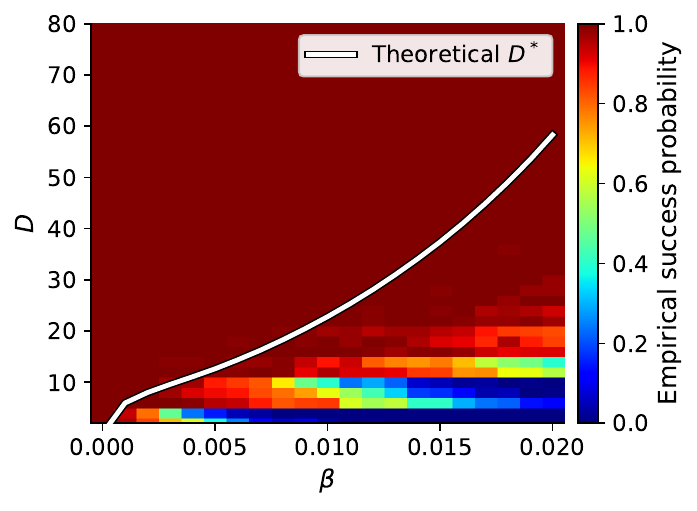}
        }\hfill
        \subfloat[Oblivious \label{fig:D_vs_beta_oblivious}]{\includegraphics[width=.45\textwidth]{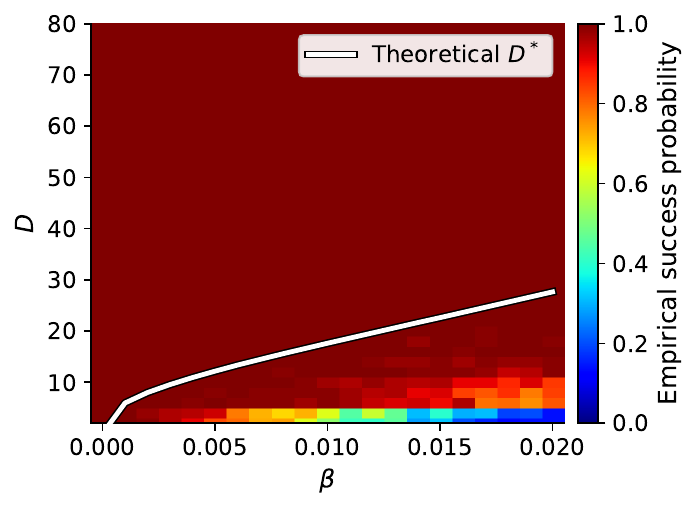}
        }\hfill
        \caption{Empirical success probability as a function of the quantile subsample size $D$ and corruption rate $\beta$, using $c_{\mathrm{succ}}=0.05$. The solid curve shows the theoretical bound $D^*$.}
        \label{fig:D_vs_beta}
\end{figure}

\section{Conclusion} \label{sec: conclusion}

We studied quantile randomized Kaczmarz (QRK) for streaming corrupted linear systems under Massart and oblivious noise. Exploiting the independence of fresh samples across iterations, we obtained explicit, computable bounds on both the maximum tolerable corruption level and the required subsample size, resolving open questions about implicit constants in prior work.
We showed that streaming QRK converges linearly to the true solution using only $D = O(\log T)$ samples per iteration, and can tolerate corruption levels up to $\beta \approx 0.069$ under Massart noise and $\beta \approx 0.32$ under oblivious noise. At $q=0.75$, $\beta=0.01$, and $T=20000$ iterations, the required subsample sizes are $D=25$ under Massart noise and $D=13$ under oblivious noise.
These theoretical guarantees can be further improved when the corruption follows a known distribution, such as Gaussian noise (see \cref{rem: Gaussian corruption}).
Our numerical simulations confirm that the derived thresholds $D^*$ track the empirical success transition closely, particularly in the Massart setting. Future directions include tightening the constants further under structured corruption, extending the analysis beyond uniform spherical measurements, and understanding how the tolerable corruption level evolves over the course of optimization rather than being fixed.

\bibliographystyle{plain}
\bibliography{ref}

\end{document}